\documentclass{elsarticle}
\usepackage[english]{babel}
\usepackage[utf8]{inputenc}
\usepackage{amssymb}
\usepackage{amsfonts}
\usepackage{amsthm}
\usepackage{amssymb}
\usepackage{color}
\usepackage{hyperref}
\usepackage{rotating}
\usepackage{amsmath}
\usepackage{graphicx}
\usepackage{latexsym}
\usepackage{array}
\usepackage{multirow}
\usepackage{textcomp}
\usepackage{subeqnarray}
\usepackage{rotating}
\usepackage{tikz}
\usepackage{color}
\usepackage{array}
\usepackage{multirow}
\usepackage{subeqnarray}
\usepackage{tikz}
\usepackage{subfigure}
\usepackage{relsize}
\usepackage{adjustbox}

\newtheorem{theorem}{Theorem}
\newtheorem{proposition}{Proposition}
\newtheorem{lemma}{Lemma}

\begin{document}

\begin{abstract}
The Angular Constrained Minimum Spanning Tree Problem ($\alpha$-MSTP) is a combinatorial optimization problem with a strong computational geometry flavor. It is defined in terms of a complete undirected graph $G=(V,E)$ and an angle $\alpha \in (0,2\pi]$. Vertices of $G$ define points in the Euclidean plane while edges, the line segments connecting them, are weighted by the Euclidean distance between their endpoints. A spanning tree is an $\alpha$-spanning tree ($\alpha$-ST) of $G$ if, for any $i \in V$, the smallest angle that encloses all line segments corresponding to its $i$-incident edges does not exceed $\alpha$. $\alpha$-MSTP consists in finding an $\alpha$-ST with the least weight. We introduce two $\alpha-$MSTP integer programming formulations, ${\mathcal F}_{xy}^*$ and $\mathcal{F}_x^{++}$ and their accompanying Branch-and-cut (BC) algorithms, {\tt BCFXY}$^*$ and {\tt BCFX}$^{++}$. Both formulations can be seen as improvements over formulations coming from the literature. The strongest of them, $\mathcal{F}_x^{++}$, was obtained by: (i) lifting an existing set of inequalities in charge of enforcing $\alpha$ angular constraints and (ii) characterizing $\alpha$-MSTP valid inequalities from the Stable Set polytope, a structure behind $\alpha-$STs, that we disclosed here. These formulations and their predecessors in the literature were compared from a polyhedral perspective. From a numerical standpoint, we observed that {\tt BCFXY}$^*$ and {\tt BCFX}$^{++}$ compare favorably to their competitors in the literature. In fact, thanks to the quality of the bounds provided by $\mathcal{F}_x^{++}$, {\tt BCFX}$^{++}$ seems to outperform the other existing $\alpha-$MSTP algorithms. It is able to solve more instances to proven optimality and to provide sharper lower bounds, when optimality is not attested within an imposed time limit. As a by-product, {\tt BCFX}$^{++}$ provided 8 new optimality certificates for instances coming from the literature.

\end{abstract} 
%===================================================================

\begin{frontmatter}

\title{Improved Formulations and Branch-and-cut Algorithms for the Angular Constrained Minimum Spanning Tree Problem}

\author{Alexandre Salles da Cunha \fnref{fn1}\corref{corr1}}
\cortext[corr1]{Corresponding author}
\fntext[fn1]{Alexandre Salles da Cunha was partially funded by CNPq grant 303928/2018-2 and FAPEMIG grant CEX - PPM-00164/17.  This manuscript was submitted to Computers \& Operations Research on April 20th, 2020.}
\address{Universidade Federal de Minas Gerais, Departamento de Ci\^encia da Computa\c c\~ao, Belo Horizonte, Brazil}
\ead{acunha@dcc.ufmg.br}

\end{frontmatter}

%=========================================================
%Introducao
%\input{introduction-v03.tex}
\section{Introduction}
\label{sec:introduction}

The Angular Constrained Minimum Spanning Tree Problem ($\alpha$-MSTP) is a combinatorial optimization problem with a strong computational geometry flavor. It is defined in terms of an angle $\alpha \in (0,2\pi]$ and a complete undirected graph $G = (V,E)$, with $n = |V|$ vertices and $m = |E|$ edges. Every vertex of $V$ corresponds to a point in the Euclidean plane. An edge $e=\{i,j\} \in E$ represents the line segment connecting $i$ and $j$. A weight $w_{e}$, corresponding to the Euclidean distance between the endpoints $i$ and $j$, is assigned to each edge $e = \{i,j\}$ of $E$. The weight of a spanning tree $(V,E_T)$ is the sum of the weights of its edges, $\sum_{e \in E_T}w_e$. A spanning tree of $G$ is an $\alpha$-spanning tree ($\alpha$-ST) if, for every vertex $i \in V$, the smallest angle enclosing all line segments corresponding to its $i$-incident edges does not exceed $\alpha$.  $\alpha-$MSTP looks for an $\alpha-$ST with the minimum possible weight. 

In order to illustrate the geometry of the $\alpha$-Angular Constraints ($\alpha-$ACs), consider the points indicated in Figure \ref{fig:fig1}. Horizontal and vertical coordinates for each vertex (or, equivalently, point in the plane) are given in Table \ref{tab:datafigura1}. Take $\alpha = \frac{\pi}{3}$, for instance. Edges $\{i,u\}$, $\{i,t\}$ and $\{i,l\}$ cannot simultaneously belong to a $\frac{\pi}{3}$-ST of $G$ since the smallest angle enclosing all these edges has $\pi$ radians.   To validate such an observation, assume that to each edge incident to $i$ corresponds an unitary vector (under the Euclidean norm) directed from $i$ towards the other endpoint of the edge.  All unitary vectors associated to the edges depicted in Figure \ref{fig:fig1} are plotted in Figure \ref{fig:fig2}. Thus, $\vec{iu}$, $\vec{it}$ and $\vec{il}$ are the unitary vectors, directed from $i$ towards the other endpoints of edges $\{i,u\}$, $\{i,t\}$ and $\{i,l\}$, respectively.  
 Note that the circular sector obtained by rotating vector $\vec{iu}$ anti-clockwise round $i$, passing by $\vec{it}$, until it reaches $\vec{il}$, involves $\pi$ radians and that exceeds the maximum value of $\alpha = \frac{\pi}{3}$. As another example, consider the edges $\{i,z\}$, $\{i,q\}$. Note that there is a circular sector of at most $\frac{\pi}{3}$ radians that covers both unitary vectors $\vec{iz}$ and $\vec{iq}$. For instance, consider the sector that starts precisely at $\vec{iz}$ and spans  $\alpha = \frac{\pi}{3}$ radians anti-clockwise round $i$. That sector ends $\frac{\pi}{6}$ radians past $\vec{iq}$ and thus contains both unitary vectors. Thus, these two edges could be included in a $\frac{\pi}{3}-$ST of $G$.
 
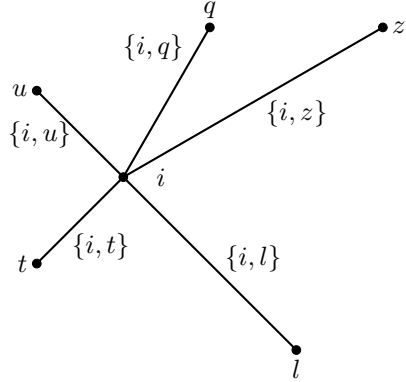
\begin{figure}[!h]
\begin{center}
% ==================================================
% \input{./Figuras/subfig1.tex}
\begin{tikzpicture}[xscale=1.15,yscale=1.15] 
\coordinate (vi) at (0,0);
\coordinate (vz) at (3,1.73205);
\coordinate (vq) at (1,1.73205);
\coordinate (vu) at (-1,1);
\coordinate (vt) at (-1,-1);
\coordinate (vl) at (2,-2);
\draw[fill] (vi) circle [radius=0.05];
\draw[fill] (vz) circle [radius=0.05];
\draw[fill] (vq) circle [radius=0.05];
\draw[fill] (vu) circle [radius=0.05];
\draw[fill] (vt) circle [radius=0.05];
\draw[fill] (vl) circle [radius=0.05];
\node [right] at (vi) {\hspace{2mm} $i$}; 
\node [right] at (vz) {$z$};
\node [above] at (vq) {$q$}; 
\node [left] at (vu) {$u$}; 
\node [left] at (vt) {$t$};
\node [below] at (vl) {$l$};  
\node [left] at (0.8,1.5) {$\{i,q\}$}; 
\node [below] at (2.0,1.0) {$\{i,z\}$};
\node [below,left] at (-.5,.5) {$\{i,u\}$};
\node [right] at (-0.7,-.8) {$\{i,t\}$};
\node [above] at (1.5,-1.2) {$\{i,l\}$};
\draw [-, thick] (vi) -- (vz);
\draw [-, thick] (vi) -- (vq);
\draw [-, thick] (vi) -- (vu);
\draw [-, thick] (vi) -- (vt);
\draw [-, thick] (vi) -- (vl);
\end{tikzpicture} 
% ==================================================
\end{center}
\caption{Euclidean plane points $i, l, q, t, u$ and $z$ (corresponding coordinates shown in Table \ref{tab:datafigura1}) and line segments connecting $i$ to every other point. Figure extracted from \cite{cunha:2019}.}
\label{fig:fig1}
\end{figure}

\begin{table}[!h]
\caption{Input data for points (i.e., vertices) in Figures \ref{fig:fig1}-\ref{fig:fig3}. Table extracted from \cite{cunha:2019}.}
\label{tab:datafigura1}       
\renewcommand{\arraystretch}{1.35}
\begin{center}
\begin{tabular}{lcccc}
\hline
 & & \multicolumn{2}{c}{Coordinates for the Euclidean plane points} & \\ \cline{3-4}
vertex of $V$ & & horizontal & vertical & $\measuredangle_{0ip}$  \\
\hline
$i$ & & 0   & 0    & $-$ \\
$z$ & & 3    & $\sqrt{3}$ & $\frac{\pi}{6}$ \\
$q$ & & 1    & $\sqrt{3}$ & $\frac{\pi}{3}$ \\
$u$ & & -1    & 1 & $\frac{3\pi}{4}$ \\
$t$ &  & -1   & -1 & $\frac{5\pi}{4}$ \\
$l$ &  & 2   & -2 & $\frac{7\pi}{4}$ \\
\hline
%\noalign{\smallskip}\hline
\end{tabular}
\end{center}
\renewcommand{\arraystretch}{1}
%\caption{\textcolor{blue}{Input} data for \textcolor{blue}{points (i.e., vertices)} in Figures \ref{fig:fig1}-\ref{fig:fig3}.}
\end{table}

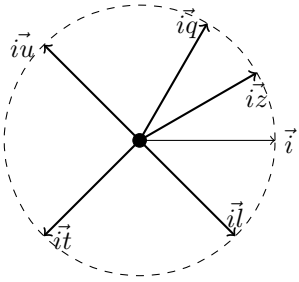
\begin{figure}[!t]
\begin{center}
% ======================================================
%. \input{./Figuras/subfig2a.tex}
\begin{tikzpicture}[xscale=1.8,yscale=1.8] 
\coordinate (vi) at (0,0);
\coordinate (vvi) at (1,0);
\coordinate (vz) at (3,1.73205);
\coordinate (vq) at (1,1.73205);
\coordinate (vu) at (-1,1);
\coordinate (vt) at (-1,-1);
\coordinate (vl) at (2,-2);
\coordinate (vzu) at (0.8660,.5);
\coordinate (vqu) at (.5,0.8660);
\coordinate (vuu) at (-0.70711,0.70711);
\coordinate (vtu) at (-0.70711,-0.70711);
\coordinate (vlu) at (0.70711,-0.70711);
\draw[fill] (vi) circle [radius=0.05];
\node [left] at (vqu) {$\vec{iq}$}; 
\node [below] at (vzu) {$\vec{iz}$};
\node [below,left] at (vuu) {$\vec{iu}$};
\node [right] at (vtu) {$\vec{it}$};
\node [above] at (vlu) {$\vec{il}$};
\node [right] at (vvi) {$\vec{i}$};
\draw [->] (vi) -- (vvi);
\draw [->, thick] (vi) -- (vzu);
\draw [->, thick] (vi) -- (vqu);
\draw [->, thick] (vi) -- (vuu);
\draw [->, thick] (vi) -- (vtu);
\draw [->, thick] (vi) -- (vlu);
\draw [dashed] (0,0) circle (1);
\end{tikzpicture} 
% ======================================================
\end{center}
%\caption{Unit vectors associated to the points depicted in Figure \ref{fig:fig1}.}
%\caption{Unit vectors associated with the points that correspond to the vertices in Figure \ref{fig:fig1}.}
\caption{Unit vectors corresponding to the line segments (i.e., edges) in Figure \ref{fig:fig1}. Figure extracted from \cite{cunha:2019}. }
\label{fig:fig2}
\end{figure}

$\alpha-$MSTP suits well as a model for the design of wireless networks that rely on directional antennas. Such antennas concentrate power in directions that span restricted angles and, because of that, have some advantages over omni-directional antennas, that irradiate power in all directions \cite{carmi:2011}. They are able to reduce energy consumption, network congestion and signal interference \cite{carmi:2011,ackerman:2013,yu:2014}. 

To illustrate how the $\alpha-$ACs affect communication in these applications, consider the points indicated in Figure \ref{fig:fig1} and and assume that one directional antenna is placed at each point indicated there. Consider as well the unitary vectors associated to the edges incident to $i$, indicated in Figure \ref{fig:fig2}. Direct communication between $i$ and any of its neighbors in $G$, say $q$, can only take place if the orientations of the directed antennas placed at $i$ and $q$ allow the signal sent by $i$ to be captured by $q$ and vice-versa. Suppose all antennas concentrate power in directions spanning angles of $\alpha = \frac{\pi}{3}$. Associated to each antenna there is a cone (of signal) that defines the angular sector where power is concentrated. The cone is  defined by the point where it is placed and by its two unitary extreme rays. If one of the rays of the cone  placed at $i$ is the vector $\vec{iz}$ and the other ray ends $\frac{\pi}{3}$ radians anti-clockwise round $i$ from $\vec{iz}$, the signal sent by $i$ can reach points $z$ and $q$ of Figure \ref{fig:fig1}. In such a simplified model, these two points capture the signal sent by $i$ no matter how far they are from the source $i$. The remaining vertices do not capture the signal sent by $i$ because they lie outside the cone. We say that points $z$ and $q$ are {\it seen} by the antenna placed at $i$. If, in addition to that, the orientation of the antenna placed at $q$ allows $i$ to capture the signal sent by $q$, $i$ and $q$ can communicate directly. Accordingly, edge $\{i,q\}$ can belong to a $\frac{\pi}{3}-$ST of the associated undirected graph.

To the best of our knowledge, Aschner and Katz \cite{aschner:2017} and Cunha and Lucena 
\cite{cunha:2019} are the only two references dedicated to the problem. Aschner and Katz \cite{aschner:2017} introduced $\alpha-$MSTP and demonstrated the NP-Completeness of its decision version for $\alpha \in \{ \frac{2\pi}{3},\pi\}$. From an algorithmic perspective, they introduced approximation methods for different values of $\alpha$ and efficient algorithms for some polynomial time solvable cases. %They also showed that $\alpha-$STs are only guaranteed to exist when $\alpha \geq \frac{\pi}{3}$ applies.

Following another line of research, Cunha and Lucena \cite{cunha:2019} introduced two different integer programming (IP) formulations for $\alpha-$MSTP: $\mathcal{F}_{x}$, a formulation defined on the {\it natural} space of edge variables $\mathbf{x} \in \{0,1\}^m$ and $\mathcal{F}_{xy}$, an {\it extended} formulation that also uses a second set of variables, $\mathbf{y}\in \{0,1\}^{2m}$. $\mathcal{F}_{xy}$ uses $\mathbf{x}$ for choosing the edges in the spanning tree and $\mathbf{y}$ to enforce the $\alpha-$ACs. Formulation $\mathcal{F}_{x}$, on the other hand, only needs the first set of variables since the $\alpha-$ACs are enforced by different modeling arguments that do not require the use of $\mathbf{y}$.  Both formulations were investigated and compared from a polyhedral point of view. The role of $\alpha$ on the relative strength of these formulations was investigated as well. A constructive heuristic and two Branch-and-cut algorithms (BC), {\tt BCFXY} and {\tt BCFX}, respectively based on formulations $\mathcal{F}_{xy}$ and $\mathcal{F}_{x}$, were also implemented and tested there. 

From now on, denote by $w({\cal F})$ the Linear Programming Relaxation (LPR) bounds provided by any $\alpha-$MSTP formulation $\mathcal{F}$. 

\subsection{Our contribution}

In this paper, we present improved formulations for $\alpha-$MSTP. The first, ${\cal F}_{xy}^*$, differs from $\mathcal{F}_{xy}$ in minor details. ${\cal F}_{xy}^*$ does not include one set of $2m$ (non-redundant) constraints that arise in the definition of  $\mathcal{F}_{xy}$. We show that $w(\mathcal{F}_{xy}) = w(\mathcal{F}_{xy}^*)$ always holds, despite the fact that $\mathcal{F}_{xy}$ may be strictly contained in $\mathcal{F}_{xy}^*$.   In addition, we show that the vector of variables $\mathbf{y}$ does not need to be integer constrained. Our second formulation, $\mathcal{F}_x^{++}$, builds on formulation $\mathcal{F}_x$ from \cite{cunha:2019} and includes some new valid inequalities, characterized here. One of these families of valid inequalities is actually a lifting of valid inequalities used in $\mathcal{F}_{x}$ to enforce the $\alpha-$ACs. Alone, the inclusion of these lifted inequalities to $\mathcal{F}_x$ lead to a much stronger formulation, $\mathcal{F}_{x}^{+}$. We showed that the projection of  $\mathcal{F}_{xy}^*$ (and $\mathcal{F}_{xy}$) onto the $\mathbf{x}$ space is contained in $\mathcal{F}_{x}^{+}$, so that $w(\mathcal{F}_{x}^{+}) \leq w(\mathcal{F}_{xy}^*)$ holds. The question of whether these LPR bounds do match is still open; in all our numerical testings, these two values were identical, though. We also characterized a Stable Set Structure in solutions for $\alpha-$MSTP, for $\alpha < \pi$. Thus, we investigated the use of valid inequalities for the Stable Set polytope \cite{padberg:1973}, to strengthen LPR bounds $w(\mathcal{F}_x^+)$ for  $\alpha-$MSTP. The resulting reinforced formulation, $\mathcal{F}_{x}^{++}$, is empirically shown to provide the strongest known $\alpha$-MSTP LPR bounds for the $\alpha$-MSTP instances tested here, the hardest in the literature.

%To summarize, \textcolor{red}{we showed that  $\mathcal{F}_x^{++} \subseteq \mathcal{F}_{xy} \subseteq \mathcal{F}_{xy}^* \subseteq \mathcal{F}_{x}^{+} \subseteq \mathcal{F}_{x}$, for $\alpha < \pi$.} $\mathcal{F}_x \subset \mathcal{F}_{x}^+ \subseteq \mathcal{F}_{xy} \subseteq \mathcal{F}_{xy}^* \subset \mathcal{F}_{x}^{++}$.

On the algorithmic side, we introduced two Branch-and-cut (BC) algorithms, {\tt BCFXY}$^*$ and {\tt BCFX}$^{++}$, respectively based on $\mathcal{F}_{xy}^{*}$ and  $\mathcal{F}_{x}^{++}$. We extended the computational experiments conducted in \cite{cunha:2019}, considering now additional values of $\alpha$ in the range $\alpha < \pi$, to which correspond the hardest instances of the problem. Our computational results suggest that, thanks to the strength of LPR bounds $w(\mathcal{F}_x^{++})$, {\tt BCFX}$^{++}$ outperforms its competitors, being able to solve more instances to proven optimality within a two hour CPU time limit. As a by product, algorithm {\tt BCFX}$^{++}$ provides 8 additional optimality certificates for instances coming from the literature.

The remainder of the paper is organized as follows. In Section \ref{sec:notation}, we complement  the notation described so far. Formulations $\mathcal{F}_{xy}$ and $\mathcal{F}_x$ from \cite{cunha:2019} are reviewed in Section \ref{sec:literature} while our improved ones are presented in Section \ref{sec:improvedformulations}. The BC algorithms {\tt BCFXY}$^{*}$ and {\tt BCFX}$^{++}$ based on these enhanced models are discussed in Section \ref{sec:bcs}. In Section \ref{sec:results}, we numerically evaluate the quality of the LPR bounds introduced here and compare four $\alpha-$MSTP BC algorithms, two introduced here and two coming from \cite{cunha:2019}. We close the paper in Section \ref{sec:conclusions}, indicating directions for future research. The paper also includes an Appendix, where the projection of $\mathcal{F}_{xy}^*$ onto the $\mathbf{x}$ space is shown to be contained in $\mathcal{F}^+$. Aggregated computational results presented in Section \ref{sec:results} are complemented with an on-line supplementary document, where more detailed computational results are offered.
%=======================================================================
%Notation
%\input{notation-v03.tex}
\section{Notation}
\label{sec:notation}

For every $S \subseteq V$, denote by $\delta(S)$ (resp. $E(S)$) the edges of $G$ with one end vertex (resp. the two end vertices) in $S$.
For simplicity, if $S$ contains a single vertex, say vertex $i$, we use $\delta(i)$ instead of $\delta(\{i\})$. Similarly, given a spanning tree $T = (V,E_T)$ of $G$, $\delta_T(i) := \delta(i) \cap E_T$ denotes its $i$-incident edges. The set of all spanning trees of $G$, satisfying the $\alpha-$ACs or not, is denoted by ${\cal T}$.

The remainder of this section is dedicated to presenting the notation used to mathematically formalize the $\alpha-$ACs. The notation used here is precisely that introduced in \cite{cunha:2019}; the figures and drawings we make use  to present the notation were extracted from \cite{cunha:2019}. The definitions that follow are illustrated for the edges and vertices depicted in Figure \ref{fig:fig1}, whose horizontal and vertical coordinates are given in Table \ref{tab:datafigura1}.

We define the set of unit vectors that are collinear and have the same directions as those vectors having $i \in V$ as their initial point and $j \in V \setminus \{i\}$ as their terminal ones as $R_i = \{\vec{ij} : j \in V \setminus \{i\} \}$. Note that vectors in set $R_i$ are directly associated to the edges of $\delta(i)$. All vectors associated to the edges in $\delta(i)$ plotted in Figure \ref{fig:fig1} are indicated in Figure \ref{fig:fig2} which also plots an additional vector, $\vec{i}$. For every $i \in V$, $\vec{i}$ is the vector equal to $(1\hspace{2mm} 0)^T$, i.e., a vector with the same magnitude and direction as $(1\hspace{2mm} 0)^T$ but whose initial point is placed at $i \in V$. We define by $d_i$ the closed unit disk centered at $i$. The boundary of $d_i$, where all terminal points of the vectors in $R_i$ are located, is also indicated in Figure \ref{fig:fig2}. 

We define $\measuredangle_{0ij}$ as the anti-clockwise angle that $\vec{i}$ forms with $\vec{ij}$. For every vector $\vec{ij} \in R_i$ in Figure \ref{fig:fig2}, corresponding angles $\measuredangle_{0ij}$ are shown in Figure \ref{fig:fig3}. The angle obtained by rotating, say $\vec{ij}$, anti-clockwise around $i$ until it becomes collinear with $\vec{ik}$ for $j \neq k$ is denoted by $\measuredangle_{jik}$. Similarly, $\measuredangle_{kij}$ is the angle obtained when $\vec{ik}$ moves anti-clockwise round $i$ until $\vec{ij}$ is met.

Angles $\measuredangle_{jik}$ and $\measuredangle_{kij}$ can be computed quite easily. To that aim, let $\measuredangle^{j}_{0ik} = \measuredangle_{0ik}$, if $\measuredangle_{0ik} \geq \measuredangle_{0ij}$, and $\measuredangle^{j}_{0ik} = 2\pi + \measuredangle_{0ik}$, if $\measuredangle_{0ik} < \measuredangle_{0ij}$ otherwise applies. It then follows that $\measuredangle_{jik} = \measuredangle^{j}_{0ik} - \measuredangle_{0ij}$.
Angle $\measuredangle_{kij}$ may be computed in a similar fashion. When $\vec{ij}$ and $\vec{ik}$ are collinear and point at the same direction, the two vectors are the same. In that case, $\measuredangle_{0ij} = \measuredangle_{0ik}$ applies. Thus, $\measuredangle_{jik} = \measuredangle_{kij} = 0$ results. Conversely, if $\vec{ij}$ and $\vec{ik}$ are different vectors, $\measuredangle_{jik} + \measuredangle_{kij} = 2\pi$ then holds. Table \ref{tab:angles1} indicates the corresponding angles $\measuredangle_{jik}$ for every ordered pair of distinct vectors $\vec{ij}$ and $\vec{ik}$ in $R_i$. In addition, the table also highlights the largest of these values, for every individual vector $\vec{ij}$ in that figure.

For a given subset $s(i) \subseteq \delta(i)$, define $Z_i \subseteq R_i$ as the subset of vectors of $R_i$ that correspond to the edges of $s(i)$. The edges in $s(i)$ satisfy the $\alpha-$AC if and only if one is able to find a circular sector for disk $d_i$, with at most $\alpha$ radians round $i \in V$, that encloses all vectors in $Z_i$. We now discuss how to compute appropriate angles and check whether or not a given set of edges $s(i)$ meet such requirements. By rotating any given $\vec{ij} \in Z_i$ anti-clockwise around $i$, the angles $\measuredangle_{jil}$ that $\vec{ij}$ forms with every $\vec{il} \in Z_i \setminus \{\vec{ij}\}$ can be easily computed. Out of these angles, assume that $\measuredangle_{jim}$ is the one with the largest value and consider the circular sector it implies in disk $d_i$. Apply the procedure to every individual $\vec{ij} \in Z_i$ and 
compute the largest of these values. After these $|Z_i|$ maximum angles were computed, one just need to keep the one with the smallest angle. Denote it by $\sphericalangle^{i}$ and its central angle by $\theta(\sphericalangle^{i})$ and note that this is the circular sector we seek. 
Computing $\sphericalangle^{i}$ thus requires the identification of one circular sector for every $\vec{ij} \in Z_i$, i.e., the one with the largest central angle $\measuredangle_{jim}$. Among these, the one with the smallest central angle then defines $\sphericalangle^{i}$. 
For the vertices indicated in Figure \ref{fig:fig1}, $s(i) = \delta(i)$ and thus $Z_i = R_i$, note that $\theta(\sphericalangle^i) = \frac{3\pi}{2}$ then results. Note that the last column in Table \ref{tab:angles1} gives the maximum angular sector for each $\vec{ij} \in R_i$.  The smallest of these values, corresponding to $\measuredangle_{tiu}$ and $\measuredangle_{lit}$, has $\frac{3\pi}{2}$ radians.

For every $i \in V$ and every $\vec{ij} \in R_i$, define
$$L_{ij} := \{j\} \cup \{k \in V \setminus \{i,j\}:  \measuredangle_{kij} \leq \alpha\}.$$
\noindent
Now note that $k \in L_{ij}$ if and only if one reaches or goes past $\vec{ij}$ by rotating $\vec{ik}$ anti-clockwise around $i$ by an angle of $\alpha$ radians.
Therefore, all vertices $k \in L_{ij}$ are such that $\measuredangle_{0ij}^k$, which is equal to either $\measuredangle_{0ij}$ or $\measuredangle_{0ij} +2\pi$, by definition, satisfies $\measuredangle_{0ij}^k \in [\measuredangle_{0ik}, \measuredangle_{0ik} + \alpha$].
For the unit vectors $R_i$ in Figure \ref{fig:fig2}, Table \ref{tab:L} shows the different sets $L_{ij}$ that apply to every different value $\alpha \in  \{ \pi, \frac{2\pi}{3}, \frac{3\pi}{2}\} $.

We can now address the central question of deciding whether a spanning tree $T = (V,E_T)$ is feasible or not, depending on the spanning tree edges incident to each of its vertices. The procedure in charge of that checks the satisfaction of $\alpha-$ACs for one vertex $i \in V$ at a time, as follows. Define $\sphericalangle^i_T$ as the smallest angled circular sector that simultaneously encloses all vectors in $Z_i = \{ \vec{ij} \in R_i: \{i,j\} \in \delta_T(i)\}$. If $i$ is a leaf of $T$, define $\theta(\sphericalangle^i_T) = 0$. Otherwise,  if $|\delta_T(i)| \ge 2$ applies, $\theta(\sphericalangle^{i}_T)$ is then given by
\begin{equation}
\label{defangle0}
\theta(\sphericalangle^{i}_T) = \min_{\{i,j\} \in \delta_T(i)} \left \{ \max_{\{i,k\} \in \delta_T(i) \setminus  \{i,j \}}\left \{ \measuredangle_{jik} \right \} \right \}.
\end{equation}
If $\theta(\sphericalangle^{i}_T) \leq \alpha$ holds for every $i \in V$, $T$ is an $\alpha$-ST and is therefore feasible for the problem. Computing $\sphericalangle^i_T$ could be made simpler, provided that the edges of $T$ incident to $i$, $\delta_T(i) = \{ \{i,v_1\},\{i,v_2\},\dots, \{i,v_p\} \}$ ($p = \lvert \delta_T(i)\rvert$), are conveniently sorted. Thus, suppose that $\measuredangle_{0iv_1} \leq \measuredangle_{0iv_2} \leq \cdots \leq \measuredangle_{0iv_p}$ holds. 
Given this sorting, note that $\vec{iv_1}$ forms a largest possible angle with $\vec{iv_p}$,  $\vec{iv_2}$ does that with $\vec{iv_1}$ and so on. Assuming that $v_0 \equiv v_p$, for simplicity, $\theta(\sphericalangle^{i}_T)$ may then be efficiently computed in $O(n)$ time complexity by direct comparison of just $p$ angles as

\begin{equation}
\label{defangle1}
\theta(\sphericalangle^i_T) = \min_{a = 1,\dots,p} \left \{ \measuredangle_{v_a i v_{a-1}} \right \}.
\end{equation}

\begin{figure}[!t]
\begin{center}
\begin{tikzpicture}[xscale=1.5,yscale=1.5] 
\coordinate (ivetor) at (1,0);
\coordinate (z)	at (0.8660254038, 0.5);
\coordinate (q)	at (0.5, 0.86603);
\coordinate (u)  at (-0.70711, 0.70711);
\coordinate (l)   at  (-0.70711, -0.70711);
\coordinate (t)  at (0.70711, -0.70711);
\coordinate (vi) at (0,0);
\node [below] at (vi) {\vspace{2mm} $i$}; 
\node [below] at (ivetor) {\vspace{2mm} $\vec{i}$};
\coordinate (vzu) at (0.8660,.5);
\coordinate (vqu) at (.5,0.8660);
\coordinate (vuu) at (-0.70711,0.70711);
\coordinate (vtu) at (-0.70711,-0.70711);
\coordinate (vlu) at (0.70711,-0.70711);
\node [right] at (vzu) {$\measuredangle_{0iz}$}; 
\node [left, above] at (vqu) {$\measuredangle_{0iq}$};
\node [left, above] at (-.8, .8) {$\measuredangle_{0iu}$};
\node [left, below] at (vtu) {$\measuredangle_{0it}$};
\node [below,right] at (vlu) {$\measuredangle_{0il}$};
\draw [->, thin] (vi)--(1,0); 
\draw [-, thin, dotted] (vi) -- (z);
\draw [-, thin, dotted] (vi) -- (q);
\draw [-, thin, dotted] (vi) -- (u);
\draw [-, thin, dotted] (vi) -- (t);
\draw [-, thin, dotted] (vi) -- (l);
\draw [->,dashed, thin] (1,0) arc (0:30:1);
\draw [->,dashed, thin] (1.0,0) arc (0:60:1.0);
\draw [->,dashed, thin] (1.0,0) arc (0:135:1.0);
\draw [->,dashed, thin] (1.0,0) arc (0:225:1.0);
\draw [->,dashed, thin] (1.0,0) arc (0:315:1.0);
\end{tikzpicture} 
% ===========================================================
\end{center}
\caption{Angles $\measuredangle_{0ij}$ defined by every different vector $\vec{ij}$ in Figure \ref{fig:fig2}, with horizontal axis support for $\vec{i}$ depicted for reference. Figure extracted from \cite{cunha:2019}. }
\label{fig:fig3}
\end{figure}
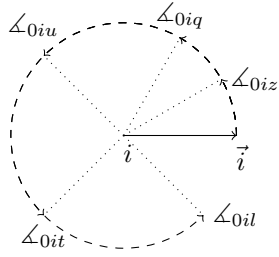

\begin{table}[!t]
\caption{Angles $\measuredangle_{jik}$ computed for every distinct pair of ordered vectors $\vec{ij}$ and $\vec{ik}$ in Figure \ref{fig:fig2}, with a highlight for the largest corresponding angle for every individual vector, $\vec{ij}$. Table extracted from \cite{cunha:2019}. }
\label{tab:angles1}
\renewcommand{\arraystretch}{1.35}
\begin{center}
\begin{tabular}{lcccccccc}
%\hline\noalign{\smallskip}
\hline
$j$   & & \multicolumn{5}{c}{$k$} & &  \\
 \cline{1-1}   \cline{3-7}
   & & $z$ & $q$ & $u$ & $t$ & $l$ & & Largest $\measuredangle_{jik}$ \\  \cline{9-9}
$z$  & & - & $\frac{\pi}{6}$ & $\frac{7\pi}{12}$ &  $\frac{13\pi}{12}$ & $\frac{19\pi}{12}$ & & $\frac{19\pi}{12}$\\
$q$  & & $\frac{11\pi}{6}$ & - & $\frac{5\pi}{12}$ &  $\frac{11\pi}{12}$ & $\frac{17\pi}{12}$ & & $\frac{11\pi}{6}$ \\
$u$  &  & $\frac{17\pi}{12}$ & $\frac{19\pi}{12}$ & - &  $\frac{\pi}{2}$  & $\pi$ & & $\frac{19\pi}{12}$ \\
$t$   &   & $\frac{11\pi}{12}$ & $\frac{13\pi}{12}$ & $\frac{3\pi}{2}$ & -  & $\frac{\pi}{2}$ & & $\frac{3\pi}{2}$\\
$l$   &   & $\frac{5\pi}{12}$& $\frac{7\pi}{12}$ & $\pi$ &  $\frac{3\pi}{2}$& - & & $\frac{3\pi}{2}$ \\
\hline
%\noalign{\smallskip}\hline
\end{tabular}
\end{center}
\renewcommand{\arraystretch}{1}
\end{table}

\begin{table}[!h]
\caption{Subsets of vertices $L_{ij}$ for every vector $\vec{ij} \in R_i$ in Figure \ref{fig:fig2} and every angle $\alpha$ in $\{ \pi, \frac{2\pi}{3}, \frac{3\pi}{2} \}$. Table extracted from \cite{cunha:2019}. }
\label{tab:L}
\renewcommand{\arraystretch}{1.35}
\begin{center}
\begin{tabular}{lcccc}
%\hline\noalign{\smallskip}
\hline
        &  & \multicolumn{3}{c}{$\alpha$} \\
             \cline{3-5}
$j$   & & $\pi$ & $\frac{2\pi}{3}$ &  $\frac{3\pi}{2}$ \\
$z$  & &   $\{t,l,z\}$ & $\{l,z \}$  & $\{u,t,l,z \}$            \\
$q$  & &   $\{l,z,q \}$  & $\{l,z,q \}$    & $\{t,l,z,q \}$             \\
$u$  &  &  $\{l,z,q,u\}$ & $\{z,q,u\}$  & $\{t,l,z,q,u\}$ \\
$t$   &   &  $\{q,u,t \}$ & $\{u,t \}$ & $\{l,z,q,u,t \}$                \\
$l$   &   &  $\{u,t,l \}$  & $\{t,l \}$   & $\{q,u,t,l \}$            \\
\hline
%\noalign{\smallskip}\hline
\end{tabular}
\end{center}
\renewcommand{\arraystretch}{1}
\end{table}
%=======================================================================
%Revisao da literatura
% \input{literature-v03.tex}

\section{IP formulations coming from the literature}
\label{sec:literature}

Given the definitions provided earlier, $\alpha-$MSTP consists in the following combinatorial optimization problem
\begin{align}
\min  \sum_{e \in  E_T}w_e \nonumber \\
  (V,E_T) \in {\cal T} \label{stform} \\
  \theta(\sphericalangle_T^i)  \leq \alpha & & \: i \in V. \label{eq:gangle0}
\end{align}

Two $\alpha-$MSTP IP formulations, $\mathcal{F}_{xy}$ and $\mathcal{F}_{x}$, were introduced in \cite{cunha:2019}. They use a binary vector $\mathbf{x} = \{x_e \in \mathbb{B}: e \in E\}$ of decision variables for selecting spanning tree edges. If $x_e = 1$, $e \in E$ is chosen for the spanning tree and $x_e = 0$ applying otherwise. To enforce
the spanning tree structure, both formulations impose that $\mathbf{x} \in \mathcal{F}_{\mathcal{T}}$, the latter being the polytope defined by the intersection of constraints \eqref{eq:cardtree}-\eqref{eq:xbound}.  It is widely known that $\mathcal{F}_{\mathcal{T}}$ in an integer polytope whose extreme points give the incidence vectors of the spanning trees of $G$ \cite{edmonds:1971}.

\begin{align}
\sum_{e \in E}x_e & = n - 1 & & \label{eq:cardtree} \\
\sum_{e \in E(S)}x_e & \leq |S| - 1 & & S \subset V, S \not = \emptyset \label{eq:sec} \\
x_e & \geq 0 & & e=\{i,j\} \in E. \label{eq:xbound} 
\end{align}

\noindent  The two formulations $\mathcal{F}_{xy}$ and $\mathcal{F}_{x}$ differ in how the generic $\alpha-$ACs \eqref{eq:gangle0} are represented by linear inequalities and by the use or not of an additional set of variables to enforce them. 

\subsection{Formulation ${\cal F}_{xy}$}

Formulation $\mathcal{F}_{xy}$ uses a second vector of decision variables, $\mathbf{y} = \{y_{ij} \in \mathbb{B}: i,j \in V, \: \vec{ij} \in R_i\}$, to enforce the $\alpha-$ACs. The role played by these variables can be easily understood if one resorts to the $\alpha-$MSTP application we highlighted before. These variables aim at aligning the directional antennas, allowing applicable pairs of points to communicate directly. The positioning of each antenna round $i$ can be represented by two extreme rays. The role of variables $\mathbf{y}$ is thus to locate the first of the two extreme rays for each $i \in V$. Accordingly, $y_{ij} = 1$ implies that the antenna placed at $i$ has the first of its two extreme transmission rays collinear with $\vec{ij}$. The other extreme ray is then positioned, anti-clockwise, $\alpha$ radians away from $\vec{ij}$. The antenna then concentrates power in the sector $[\measuredangle_{0ij},\measuredangle_{0ij} + \alpha]$.  In that case, an edge $\{i,l\}$ can be included in the spanning tree if $i$ {\it can see} $l$ (i.e., $\measuredangle_{0il}^j \in [\measuredangle_{0ij},\measuredangle_{0ij} + \alpha]$) and vice-versa.

Formulation $\mathcal{F}_{xy}$ is defined as the intersection of $\mathcal{F}_{\mathcal{T}}$ and 
\begin{align}
\sum_{\vec{ij} \in R_i}y_{ij} & = 1 & & i \in V \label{eq:assigny} \\
y_{ij} & \leq x_e & & i \in V, e=\{i,j\} \in \delta(i)               \label{eq:couplexy1} \\
x_e & \leq \sum_{k \in L_{ij}} y_{ik} & & i \in V, e=\{i,j\} \in \delta(i) \label{eq:allowed1} \\
y_{ij} & \geq 0 & & i \in V, \: \vec{ij} \in R_i. \label{eq:ybound}
\end{align}

Constraints \eqref{eq:assigny} impose that precisely one vector in $R_i$ defines the positioning of the first antenna ray, for every $i \in V$. Constraints \eqref{eq:allowed1} define which edges are admissible, depending on which variables $\mathbf{y}$ were activated. Constraints \eqref{eq:couplexy1} also couple variables $\mathbf{y}$ and $\mathbf{x}$. They state that $y_{ij}$ cannot be activated unless $x_{ij} = 1$ holds.

Cunha and Lucena \cite{cunha:2019} introduced the following $\alpha-$MSTP formulation\begin{equation}
\label{modelxy}
w^* = \min \left \{ \sum_{e \in E} w_e x_e: (\mathbf{x},\mathbf{y}) \in {\cal F}_{xy} \cap \mathbb{B}^{3m} \right \},
\end{equation}
\noindent that explicitly enforces variable $\mathbf{y}$ to be integer constrained.

\subsection{Formulation ${\cal F}_{x}$}

Differently from formulation  ${\cal F}_{xy}$, ${\cal F}_{x}$ redefines the generic $\alpha$-AC, \eqref{eq:gangle0}, as a set of exponentially many valid inequalities, solely based on variables $\mathbf{x}$. To explain how, 
consider a subset of edges $s(i) \subseteq \delta(i)$, $s(i) = \{ \{i,v_1\},\dots,\{i,v_p\} \}$, $p = |s(i)|$, indexed so that
\begin{equation}
\label{ordering}
\measuredangle_{0iv_1} \leq \cdots \leq \measuredangle_{0iv_p}
\end{equation}
applies and $v_0 \equiv v_p$ is used, for convenience. Cunha and Lucena \cite{cunha:2019} proved that if
\begin{equation}
\label{condgeral}
\bigwedge_{k=1}^{p}  \left\{ \measuredangle_{v_kiv_{k-1}} > \alpha \right\}
\end{equation}
\noindent or equivalently, if $\left \{v_1 \not \in L_{iv_p} \land v_2 \not \in L_{iv_1} \land \cdots \land v_{p} \not \in L_{iv_{p-1}}\right \}$ holds,
\begin{equation}
\label{restgeral}
|\delta_T(i) \cap s(i)| \leq |s(i)|-1
\end{equation}
is then valid for every $\alpha$-ST, since no angular sector of $\alpha$ radians encloses all edges in $s(i)$. 
The statement \eqref{restgeral} translates into the following set of exponentially many $\alpha-$MSTP valid inequalities
\begin{equation}
\label{cover}
\sum_{ e \in s(i)}x_e \leq |s(i)|-1, \: \: s(i) \subseteq \delta(i), \: \lvert s(i) \rvert \geq 2,  \: s(i) \: \hbox{ satisfies }
\eqref{ordering}-\eqref{condgeral}.
\end{equation}

In the remainder of the text, a subset $s(i) \subseteq \delta(i)$ with at least two edges, satisfying conditions \eqref{ordering} and \eqref{condgeral}, is called {\it non-admissible}. Conversely, any given set $s(i)$ satisfying the ordering \eqref{ordering}, for which \eqref{condgeral} does not hold, is called $\alpha$-ST {\it admissible}. Subsets with just one single edge are also admissible.

Cunha and Lucena \cite{cunha:2019} proved that $\alpha-$MSTP can be formulated as
\begin{equation}
\label{modelx}
w^* = \min \left \{ \sum_{e \in E} w_e x_e: \mathbf{x} \in {\cal F}_{x} \cap \mathbb{B}^{m} \right \},
\end{equation}

\noindent where $\mathcal{F}_x$ is the intersection of $\mathcal{F}_{\mathcal{T}}$ and the exponentially many inequalities \eqref{cover}. The authors also investigated $\mathcal{F}_{x}^{23}$, a relaxation for $\mathcal{F}_x$, that restricts the use of inequalities \eqref{cover} for subsets $s(i) \subseteq \delta(i)$ with only two or three edges. The authors showed that if $\alpha < \pi$, $\mathcal{F}_{x}^{23}$ defines an $\alpha-$MSTP formulation and, in addition, $\mathcal{F}_{x}^{23} = \mathcal{F}_{x}$ also holds. 

%=======================================================================
%Formulacoes aprimoradas
%\input{improved-formulations-v02.tex}
\section{Improved $\alpha-$MSTP formulations}
\label{sec:improvedformulations}

This section presents the improved formulations $\mathcal{F}_{xy}^*$ and $\mathcal{F}_{x}^{++}$, that respectively build on formulations $\mathcal{F}_{xy}$ and $\mathcal{F}_{x}$, discussed earlier. The first part of this section suggests minor changes to $\mathcal{F}_{xy}$ that lead to an equally strong formulation, $\mathcal{F}_{xy}^*$. The second part presents new $\alpha-$MSTP valid inequalities, used to reinforce $\mathcal{F}_{x}$. Some of these new inequalities are already satisfied by $\mathcal{F}_{xy}$. Others, as our numerical results demonstrate, are not.
%=======================================================================
% Fomulacao fxy melhorada
%\input{improved-fxy-v04.tex}
\subsection{Improvements on formulation ${\cal F}_{xy}$.}

Define ${\cal F}_{xy}^*$ as the intersection of inequalities \eqref{eq:cardtree}-\eqref{eq:assigny}, \eqref{eq:allowed1}-\eqref{eq:ybound}. The improvements come from two results, to be demonstrated in the sequence: (i) variables $\mathbf{y}$ do not need to be integer constrained and (ii) the removal of inequalities \eqref{eq:couplexy1} do not impact on the LPR bounds $w({\cal F}_{xy})$. 

To address the first result, consider the following remark. For a given $(\overline{\mathbf{x}},\overline{\mathbf{y}}) \in {\cal F}_{xy}: \mathbf{x} \in \mathbb{B}^m$, define $\delta_T(i) := \{ \{i,j\} \in \delta(i): \overline{x}_{ij} = 1\}$. Then, there exists at least one edge $\{i, j\} \in \delta_T(i)$ such that $j \in \bigcap_{\{i,v\} \in \delta_T(i)}L_{iv}$ if $\alpha \geq \pi$ and, if $\alpha < \pi$ there is exactly one edge $\{i,j\} \in \delta_T(i)$, such that $j \in \bigcap_{\{i,v\} \in \delta_T(i)} L_{iv}$. To see that such an observation applies, assume the contrary, i.e., there is no $j \in \bigcap_{\{i,v\} \in \delta_T(i)} L_{iv}$. Summing up inequalities \eqref{eq:allowed1} for the edges in $\delta_T(i)$ and recalling that $(\overline{\mathbf{x}},\overline{\mathbf{y}}) \in \mathcal{F}_{xy}$, we have:
\begin{align*}
|\delta_T(i)| & \leq  \sum_{\{i,v\} \in \delta_T(i)} \sum_{ k \in L_{iv}}\overline{y}_{ik} & & \\
                   & = \sum_{\{i,v\} \in \delta_T(i)} \sum_{ k \in L_{iv}|\{i,k\} \in \delta_T(i)}\overline{y}_{ik} & & \\
                & \leq (|\delta_T(i)-1)| \sum_{ k \in \bigcup_{\{i,v\} \in \delta_T(i)}L_{iv}}\overline{y}_{ik} \\
                & =  |\delta_T(i)|-1   
\end{align*}
\noindent a contradiction follows. Now note that if $\alpha < \pi$ holds, no two edges $\{i,j\}$ and $\{i,k\}$ simultaneously satisfy $\measuredangle_{jik} \leq \alpha$ and $\measuredangle_{kij} \leq \alpha$. Thus, there is exactly one vertex $j \in \bigcap_{\{i,v\} \in \delta_T(i)} L_{iv}$, for the $\alpha < \pi$ case. We can now state the following two propositions.

\begin{proposition}
\label{prop:case1}
If $\alpha < \pi$ and $(\overline{\mathbf{x}},\overline{\mathbf{y}}) \in {\cal F}_{xy}: \overline{\mathbf{x}} \in \mathbb{B}^m$, then $\overline{\mathbf{y}} \in \mathbb{B}^{2m}$. %In other words, integrality of $\mathbf{y}$ comes as a consequence of the integrality of $\mathbf{x}$, when $\alpha < \pi$. Thus, vector $\mathbf{y}$ does not need to be integer constrained when $\alpha < \pi$.
\end{proposition}
\begin{proof}
%Denote $\delta_T(i) := \{ \{i,j\} \in E: \overline{x}_{ij} = 1\}$.
%Due to \eqref{eq:assigny} and \eqref{eq:couplexy1}, we have for $\{i,v\} \in \delta_T(i)$:
%\begin{align*}
%1 = \overline{x}_{iv} & \leq \sum_{ k \in L_{iv}}\overline{y}_{ik} & & \\
%& = \sum_{ k \in L_{iv}|\{i,k\} \in \delta_T(i)}\overline{y}_{ik} & &
%\end{align*}
Pick the edge $\{i,v_1\} \in \delta_T(i)$ such that $v_1 \in   \bigcap_{\{i,v\} \in \delta_T(i)}L_{iv}$. Now suppose $\overline{y}_{iv_1} \not = 1$. Again, summing up inequalities \eqref{eq:allowed1} for each $\{i,v\} \in \delta_T(i)$ we have: 
\begin{align*}
|\delta_T(i)| & \leq \sum_{\{i,v\} \in \delta_T(i)} \sum_{ k \in L_{iv}|\{i,k\} \in \delta_T(i)}\overline{y}_{ik} & & \\
                & \leq |\delta_T(i)| \overline{y}_{iv_1} + (|\delta_T(i)|-1)\sum_{ k \in \left ( \bigcup_{\{i,v\} \in \delta_T(i)}L_{iv}\right ) \setminus \{v_1\} }\overline{y}_{ik} \\
                & = |\delta_T(i)| \overline{y}_{iv_1} + (|\delta_T(i)|-1)(1- \overline{y}_{iv_1})\\
                & = |\delta_T(i)| - 1 + \overline{y}_{iv_1} \\
              & < |\delta_T(i)|
\end{align*}
\noindent and we have a contradiction. %Note that the second inequality holds because as $\alpha < \pi$, all the edges $\{i,u\} \in \delta_T(i) \setminus \{ i,v_1\}$ are such that $u \not \in  \bigcap_{\{i,v\} \in \delta_T(i)}L_{iv}$. 
Thus, $\overline{y}_{iv} = 0$ for any $\vec{iv} \in R_i \setminus \{ \vec{iv_1} \}$ and $(\overline{\mathbf{x}},\overline{\mathbf{y}}) \in {\cal F}_{xy} \cap \mathbb{B}^{3m}$.
\end{proof}

\begin{proposition}
\label{prop:case2}
If $\alpha \geq \pi$ and $(\overline{\mathbf{x}},\overline{\mathbf{y}}) \in {\cal F}_{xy}: \overline{\mathbf{x}} \in \mathbb{B}^m$, then there exists $\hat{\mathbf{y}} \in \mathbb{B}^{2m}$ such that $(\overline{\mathbf{x}},\hat{\mathbf{y}}) \in {\cal F}_{xy} \cap \mathbb{B}^{3m}$. 
\end{proposition}
\begin{proof}
Pick $\{i,v_1\}$ such that $v_1 \in   \bigcap_{\{i,v\} \in \delta_T(i)}L_{iv}$. Now set $\hat{y}_{iv_1} = 1$ and $\hat{y}_{ij} = 0$ for all vectors $\vec{ij}$ such that $\{i,j\} \in \delta(i) \setminus \{i,v_1\}$ and the result follows.
\end{proof}

As a consequence of Propositions \ref{prop:case1} and \ref{prop:case2}, any vector $\overline{\mathbf{x}} \in \mathbb{B}^m$  for which there is an associated $\overline{\mathbf{y}} \in [0,1]^{2m}$ satisfying $(\overline{\mathbf{x}},\overline{\mathbf{y}}) \in {\cal F}_{xy}$ defines the incidence vector of an $\alpha-$ST of $G$. Thus, variables $\mathbf{y}$ do not need to be explicitly enforced to assume only integer values.

Before showing the second result, consider the case where $\alpha = \pi$, $s(i) = \delta(i) = \{ \{i,j\}, \{i,k\} \}$ and $\measuredangle_{jik} = \measuredangle_{kij} = \pi$. Consider the point $(\overline{\mathbf{x}},\overline{\mathbf{y}})$ such that  $\overline{x}_{ij} = 1-\epsilon, \overline{x}_{ik} = \epsilon$ and $\overline{y}_{ij} = \epsilon, \overline{y}_{ik} = 1 - \epsilon$ for any $\epsilon \in [0,1]$. Note that the point does satisfy constraints \eqref{eq:assigny}, \eqref{eq:allowed1}-\eqref{eq:ybound} and $s(i)$ is a set of admissible edges.  However,  constraint $y_{ik} \leq x_{ik}$ is violated if $\epsilon \in [0,\frac{1}{2})$ and constraint $y_{ij} \leq x_{ij}$ is violated if $\epsilon \in (\frac{1}{2},1]$. The example thus shows that $\mathcal{F}_{xy}$ may be strictly contained in $\mathcal{F}_{xy}^*$. 

\begin{proposition} For any $(\overline{\mathbf{x}},\overline{\mathbf{y}}) \in {\cal F}_{xy}^*$, there is a vector $\hat{\mathbf{y}} \in [0,1]^{2m}$ such that $(\overline{\mathbf{x}},\hat{\mathbf{y}}) \in {\cal F}_{xy}$ and thus $w({\cal F}_{xy}) = w({\cal F}_{xy}^*)$ applies. 
\end{proposition}
\begin{proof}
We provide a constructive proof, based on the algorithm below. The algorithm receives $i \in V$ and $\overline{\mathbf{y}}_i := \{ \overline{y}_{ij}: \vec{ij} \in R_i\}$, as input, and outputs vector $\hat{\mathbf{y}}_i$. After calling the algorithm $n$ times, each one for a different $i \in V$, the vector $\hat{\mathbf{y}} := (\hat{\mathbf{y}}_1,\dots,\hat{\mathbf{y}}_n)$ has the desired property. \\ 

Algorithm: 
\begin{enumerate}
\item (Early termination checking) If there is no edge $\{i,j\} \in \delta(i): \overline{y}_{ij} > \overline{x}_{ij}$, then stop. Otherwise, move on to the next step.
\item (Initialization) $\hat{y}_{ij} \leftarrow \overline{y}_{ij}: \vec{ij} \in R_i$
\item (Renaming edges) Define $\{i,j\} := \arg\max_{\{i,u\} \in \delta(i)} \{ \overline{y}_{iu} - \overline{x}_{iu}\}$ and do:
\begin{enumerate}
\item Rename edge $\{i,j\}$ as $\{i,v_1\}$.
\item Starting from $\{i,v_1\}$, rotate anti-clockwise round $i$ and re-name the remaining edges in $\delta(i)$ as $\{i,v_2\},\{i,v_3\},\cdots,\{i,v_p\}$, so that $$\measuredangle_{0iv_1}^{v_2} \leq \measuredangle_{0iv_1}^{v_3} \leq \cdots \leq \measuredangle_{0iv_{1}}^{v_p},$$ where $p = |\delta(i)|$. For convenience, denote $\{i,v_{p+1}\} = \{i,v_1\}$.
\end{enumerate}
\item For each $k = 1,\dots,p$ do
\begin{enumerate}
\item Calculate $\eta^k = \max \{0,\hat{y}_{iv_k} - \overline{x}_{iv_k}\}$.
\item Update:
\begin{align}
\hat{y}_{iv_k} & \leftarrow \hat{y}_{iv_k} - \eta^k \\
\hat{y}_{iv_{k+1}} & \leftarrow \hat{y}_{iv_{k+1}} + \eta^k
\end{align}
\end{enumerate}
\item Restore the original edge names and vector $\hat{\mathbf{y}}_i$ accordingly. 
\end{enumerate}

The very first observation is that the algorithm above does not change the $\overline{\mathbf{x}}$ component of the solution. Thus, the costs of $(\overline{\mathbf{x}},\hat{\mathbf{y}})$ and $(\overline{\mathbf{x}},\overline{\mathbf{y}})$ are identical. Our claim is that, after calling the algorithm above for each $i \in V$, $(\overline{\mathbf{x}},\hat{\mathbf{y}}) \in {\cal F}_{xy}$ and the result follows. To show that, consider the following arguments:  
\begin{itemize}
\item $\hat{\mathbf{y}}_i$ satisfies constraints \eqref{eq:assigny}. \\ At each of the $p$ iterations of the algorithm above, $\sum_{\vec{ij} \in R_i}\hat{y}_{ij} = 1$ holds. That applies due to the initialization and because $\eta^k$ is always added and subtracted respectively to $\hat{y}_{iv_{k+1}}$ and from $\hat{y}_{iv_{k}}$.  
\item $(\overline{\mathbf{x}}_i,\hat{\mathbf{y}}_i)$ satisfies constraints \eqref{eq:couplexy1}. \\ It is quite clear that after the updates in step 4-(b), $\hat{y}_{iv_k} \leq \overline{x}_{iv_k}$  for each $k = 1,\dots,p-1$. 
We now show that after the $p-$th update in step 4-(b), when $\hat{y}_{iv_{p+1}}$ is updated (and thus when $\hat{y}_{iv_{1}}$ is updated for the second time), we have $\hat{y}_{iv_1} = \hat{y}_{iv_{p+1}} = \overline{x}_{iv_1}$ since $\eta^p = 0$.  We have that $\sum_{k = 1}^p\overline{y}_{iv_k} = 1$ and, because $\overline{\mathbf{x}} \in \mathcal{F}_{\mathcal{T}}$, $\sum_{k=1}^{p} {\overline{x}}_{iv_k} \geq 1$ holds (undirected cutset constraints are satisfied when SECs are satisfied). Thus, in the last update, when $k = p$, there is no excess $\max\{0,\hat{y}_{iv_p}-\overline{x}_{iv_p}\}$ left to be transferred and $\eta^p = 0$. In fact,  at the $p-$th iteration, $\hat{y}_{iv_1}$ is actually not updated, remaining at its previous value.
\item $(\overline{\mathbf{x}}_i,\hat{\mathbf{y}}_i)$ satisfies constraints \eqref{eq:allowed1}. \\ 
Consider three consecutive vectors in $R_i$: $\vec{iv_k}, \vec{iv}_{k+1}, \vec{iv}_{k+2}$. Note that if $v_k \in L_{iv_{k+2}}$, then $v_{k+1} \in L_{iv_{k+2}}$. Now notice that starting with $k = 1$, at each of the $p$ updates, the excess $\max \{0, \hat{y}_{iv_k} - \overline{x}_{iv_k}\}$ with respect to $\overline{x}_{iv_k}$ is always pushed anti-clockwise round $i$, to the next variable $\hat{y}_{iv_{k+1}}$. Thus, if the amount $\sum_{u \in L_{iv_k}}\overline{y}_{iv_u}$ at least matched $\overline{x}_{iv_k}$ because $(\overline{\mathbf{x}},\overline{\mathbf{y}}) \in {\cal F}_{xy}$, the amount  $\sum_{u \in L_{iv_u}}\hat{y}_{iv_k}$ could never become smaller than $\overline{x}_{iv_k}$ and inequalities \eqref{eq:allowed1} are satisfied.
\end{itemize}
\end{proof}

Assume that $\epsilon \in (\frac{1}{2},1]$ and consider the point $(\overline{\mathbf{x}},\overline{\mathbf{y}})$ provided above.   After the application of the algorithm, one would have: $\eta^1 = 2\epsilon - 1 > 0$, $\eta^2 = 0$, $\hat{y}_{ij} = 1 - \epsilon = \overline{x}_{ij}$ and $\hat{y}_{ik} = \epsilon = \overline{x}_{ij}$. Note that $(\overline{\mathbf{x}},\hat{\mathbf{y}}) \in \mathcal{F}_{xy}$.

As a consequence of the results presented in this section, one can solve $\alpha-$MSTP by solving the following mixed integer program:
\begin{equation}
\label{modelxystar}
w^* = \min \left \{ \sum_{e \in E} w_e x_e: (\mathbf{x},\mathbf{y}) \in {\cal F}_{xy}^* \cap (\mathbb{B}^{m} \times \mathbb{R}^{2m}) \right \}.
\end{equation}

%=======================================================================
% Formulacao fx melhorada
%\input{improved-fx-v03.tex}

\subsection{New $\alpha-$MSTP valid inequalities.}

In this section, we present new $\alpha-$MSTP valid inequalities. All of them are written in the $\mathbf{x}$ space and are used to strengthen formulation $\mathcal{F}_x$. The first family of valid inequalities is discussed in Subsection \ref{sec:lifting}. We show that these inequalities are satisfied by the projection of ${\cal F}_{xy}^*$ (and thus of ${\cal F}_{xy}$) onto the $\mathbf{x}$ space. The second family of valid inequalities was characterized by identifying a Stable Set structure in $\alpha-$ STs. Our computational results presented later on in the paper show  that inequalities in the second set are not satisfied by points in ${\cal F}_{xy}$, since, bounds $w(\mathcal{F}_{x}^{++})$ exceed $w(\mathcal{F}_{xy})$ counterparts, for instances in our test bed.

\subsubsection{Lifting valid inequalities \eqref{cover}}
\label{sec:lifting}

The first family of valid inequalities presented here is a lifting of inequalities \eqref{cover}. Given $\alpha$, $i \in V$ and $s(i) \subseteq \delta(i), s(i) \not = \emptyset$, define 
\begin{equation*}
v_{ij}^{(s(i),\alpha)} := \left | \left\{ \{i,k\} \in s(i): \measuredangle_{jik} \leq \alpha \right \} \right |
\end{equation*}
\noindent as the number of edges of $s(i)$ {\it covered} by $\{i,j\}$ and 
\begin{equation*}
v^{(s(i),\alpha)} := \max \{ v_{ij}^{(s(i),\alpha)}: \{i,j\} \in s(i)\}
\end{equation*}
\noindent as the maximum number of edges of $s(i)$ covered by an edge of $s(i)$. Quite clearly, $v^{(s(i),\alpha)}$ gives the maximum number of edges of $s(i)$ that can be included in an $\alpha-$ST. To check that, consider the vector $\vec{iu}$ associated to the edge
\begin{equation*}
\{i,u\} \in \arg \max \{ v_{ij}^{(s(i),\alpha)}: \{i,j\} \in s(i)\}.
\end{equation*}
The angular sector that starts at $\vec{iu}$ and rotates $\alpha$ radians anti-clockwise round $i$  encloses precisely $v^{(s(i),\alpha)}$ edges of $s(i)$. Thus, inequalities 
\begin{equation}
\label{eq:strengthen1}
\sum_{e \in s(i)}x_e \leq v^{(s(i),\alpha)}
\end{equation}
\noindent are valid for $\alpha-$MSTP. Notice that for any subset of edges $s(i): |s(i)| \geq 2$ satisfying \eqref{ordering}-\eqref{condgeral}, inequalities \eqref{eq:strengthen1} are at least as strong as \eqref{cover} since $v^{(s(i),\alpha)} \leq |s(i)|-1$. 

Inequalities \eqref{eq:strengthen1} can be lifted to a stronger form, as follows. Consider a set $s(i) \subset \delta(i)$ and an edge $\{i,u\} \in \delta(i) \setminus s(i)$. If $v^{(s(i) \cup \{i,u\},\alpha)} \leq v^{(s(i),\alpha)}$ holds, then
\begin{equation}
\label{eq:strengthen2}
x_{iu} + \sum_{e \in s(i)}x_e \leq v^{(s(i),\alpha)}
\end{equation}
is also valid for $\alpha-$MSTP and is stronger than \eqref{eq:strengthen1}. 

To illustrate the difference among inequalities \eqref{cover}, \eqref{eq:strengthen1} and \eqref{eq:strengthen2}, consider the edges in Figure \ref{fig:fig1} and their unitary vectors indicated in Figure \ref{fig:fig2}. Define $s(i) = \{ \{i,z\}, \{i,u\}, \{i,t\} \}$ and note that $\measuredangle_{ziu} = \frac{7\pi}{12}, \measuredangle_{uit} = \frac{\pi}{2}, \measuredangle_{tiz} = \frac{11\pi}{12}$. Considering the $\alpha = \frac{\pi}{3}$ case, inequality \eqref{cover} then reads $x_{iz} + x_{iu} + x_{it} \leq 2$, since the set of edges $s(i)$, ordered according to \eqref{ordering}, satisfy \eqref{condgeral}. Note that $v_{ik}^{(s(i),\frac{\pi}{3})} = 1$ for every $\{i,k\} \in s(i)$, and thus inequality \eqref{eq:strengthen1} reads as $x_{iz} + x_{iu} + x_{it} \leq 1$.  It turns out that the latter inequality can be lifted to $x_{il} + x_{iz} + x_{iu} + x_{it} \leq 1$ since $v^{(s(i) \cup \{i,l\},\frac{\pi}{3})} = v^{(s(i),\frac{\pi}{3})} = 1$. No additional strengthening can be carried out, since $v^{(s(i) \cup \{i,l\} \cup \{i,q\},\frac{\pi}{3})} = 2$, as $\measuredangle_{ziq} = \frac{\pi}{6} < \frac{\pi}{3} = \alpha$. 

Consider thus the following family of $\alpha-$MSTP valid inequalities, named Lifted Angular Constraints (LACs):
\begin{align}
\sum_{e \in s(i)}x_e & \leq v^{(s(i),\alpha)} & & s(i) \subseteq \delta(i) \hbox{ such that }\label{eq:strenthen3} \\
                     &                        & & v^{(\{i,k\} \cup s(i),\alpha)} > v^{(s(i),\alpha)} \nonumber \\
                     &                        & & \hbox{for every } \{i,k\} \in \delta(i) \setminus s(i).  \nonumber
\end{align}

Given an initial subset of edges $s(i)$, inequalities \eqref{eq:strengthen2} can be strengthened to the form \eqref{eq:strenthen3} by a sequential lifting procedure. The procedure picks $\{i,k\} \not \in s(i)$ and checks whether or not $v^{(\{i,k\} \cup s(i),\alpha)} \leq v^{(s(i),\alpha)}$ holds. In case the latter condition applies, the set $s(i)$ is updated to $s(i) \cup \{i,k\}$, otherwise, $\{i,k\}$ is discarded and another edge in $\delta(i) \setminus s(i)$ is investigated. The procedure stops when $v^{(\{i,k\} \cup s(i),\alpha)} > v^{(s(i),\alpha)}$ for every edge $\{i,k\} \in \delta(i) \setminus s(i)$.

Define formulation $\mathcal{F}_x^+$ as the intersection of $\mathcal{F}_{\mathcal{T}}$ and the exponentially many inequalities \eqref{eq:strenthen3}. Clearly, $\mathcal{F}_{x}^+ \subseteq \mathcal{F}_x$. Our computational results show that, in practice $\mathcal{F}_{x}^+$ is much stronger than $\mathcal{F}_x$ for the hardest $\alpha-$MSTP instances, i.e., when $\alpha < \pi$. However, $\mathcal{F}_{x}^+$ cannot be stronger than $\mathcal{F}_{xy}^*$ (or, stronger than $\mathcal{F}_{xy}$). That applies because $\hbox{Proj}_x(\mathcal{F}_{xy}^*)$, the projection of $\mathcal{F}_{xy}^*$ onto the $\mathbf{x}$ space, is contained in the set $\mathcal{F}_x^+$. The following result, whose proof is provided in the \ref{sec:apptheo}, summarizes these observations.

\begin{theorem}
\label{theoprojection}
$\hbox{Proj}_x(\mathcal{F}_{xy}^*) \subseteq \mathcal{F}_x^+$ and thus $w(\mathcal{F}_x^+)\leq w(\mathcal{F}_{xy}^*)$ applies.
\end{theorem}

\subsubsection{Valid $\alpha-$MSTP inequalities from the Stable Set Polytope}
\label{sec:stable}

We now discuss valid inequalities for $\alpha-$MSTP obtained from the Stable Set Polytope \cite{padberg:1973}. To that aim, consider the Independent Set $$\mathcal{I} := \{ \mathbf{x} \in \mathbb{B}^m:  \mathbf{x} \hbox{ satisfies LACs } \eqref{eq:strenthen3} \}.$$  A superset of $\mathcal{I}$ is used here to reinforce $\alpha-$MSTP LPR relaxation bounds. More specifically, we consider the subset of inequalities \eqref{eq:strenthen3}, for which $v^{(s(i),\alpha)} = 1$ applies, i.e.,  we consider valid inequalities for the set 
\begin{align*}
\mathcal{I}_1  & := \left \{ \mathbf{x} \in \mathbb{B}^m: \mathbf{x} \hbox{ satisfies LACs }  \eqref{eq:strenthen3}, \right. \\
                      & \left.   \hspace{5mm} \hbox{ defined for } s(i) \subseteq \delta(i): \: v^{(s(i),\alpha)} = 1, \hbox{ for all } i \in V \right \}.
\end{align*}

Note that if $\alpha \geq \pi$, no pair of adjacent edges of $E$, say $\{i,j\}$ and $\{i,u\}$, are non-admissible. Quite clearly, either $\measuredangle_{jiu} \leq \pi$ or else 
$\measuredangle_{uij} \leq \pi$ holds in that case. Thus, $\mathcal{I}_1 = \{0,1\}^m$  when $\alpha \geq \pi$. Hence, the set of valid inequalities derived next are of help for improving $\alpha-$MSTP LPR relaxation bounds only for the $\alpha < \pi$ case, the hardest ones, as demonstrated in \cite{cunha:2019}. For the $\alpha \geq \pi$ case, valid inequalities for the more general Independent Set $\mathcal{I}$ could be useful. These, however, are not investigated in this study.

In order to characterize additional valid inequalities for $\alpha-$MSTP, consider the conflict graph \cite{atamturk:2000} $G_c = (V_c,E_c)$ associated to $\mathcal{I}_1$. Each edge of $E$ gives rise to a vertex in $V_c$. For any given $i \in V$, a pair $e = \{i,j\}, f= \{i,u\}$ of distinct edges adjacent to $i$ gives rise to an edge  $\{e,f\} \in E_c$  if the pair is non-admissible, i.e., if $\measuredangle_{jiu} >\alpha$ and $\measuredangle_{uij} >\alpha$. Thus $x_e + x_f \leq 1$ is valid for $\alpha-$MSTP. Such an inequality is a particular case, possibly a weakened version, of inequalities that define $\mathcal{I}_1$.

Among the various known classes of valid inequalities for the Stable Set polytope, we consider odd-cycles and odd-holes. Let $C \subseteq E_c$ denote an odd-cycle in $G_c$, i.e., a simple cycle of $G_c$  for which the number of vertices visited in the cycle (or edges traversed in the cycle), $|C|$, is odd.  Cycle inequalities
\begin{equation}
\label{eq:cycleineq}
\sum_{e \in C}x_e \leq \lfloor \frac{|C|}{2} \rfloor
\end{equation}
\noindent are valid for $\mathcal{I}_1$ and for $\alpha-$MSTP. A chordless cycle of $G_c$ is called hole of $G_c$. Nemhauser and Trotter \cite{nemhauser:1974} have shown that when $C$ is an odd-hole, inequalities \eqref{eq:cycleineq} are facet defining for the convex hull of points in the associated Stable Set polytope. Although the result does not directly extend to the $\alpha-$MSTP case, odd-holes are preferable to odd-cycles, since the former are stronger than the latter (see the discussion on that matter in Section 3.1 of \cite{rebennack:2012}). Details on the separation of odd-cycles (holes) will be provided later on in the paper.

Several other families of valid inequalities for the Stable Set Problem were characterized in the literature, for instance: clique inequalities, web and wheel inequalities (see \cite{rebennack:2012} for a survey). Nevertheless, these additional inequalities were not used here, since either their separation time was too high to be used, or else, they were not effective in strengthening $\alpha-$MSTP LPR bounds. For instance, consider the case of clique inequalities. During a pre-processing step, all maximal cliques of $G_c$ were identified with the algorithm in \cite{segundo:2018} and stored in a list, to be scanned within the BC search tree. For all instances in our test bed, no clique inequalities from the list were violated, at the root node of our BC method. Thus, their separation was not included in our final implementation.

From now on, denote by $\mathcal{F}_{x}^{++}$ the intersection of $\mathcal{F}_{x}^{+}$ with cycle inequalities \eqref{eq:cycleineq}.
 %==============================================================================
% Algoritmos branch-and-cut
%\input{bcalgorithms-v03.tex}
\section{Improved Branch-and-cut algorithms}
\label{sec:bcs}

In this section, we provide the main implementation details of two BC algorithms, {\tt BCFXY}$^*$ and {\tt BCFX}$^{++}$, respectively based on formulations $\mathcal{F}_{xy}^*$ and $\mathcal{F}_{x}^{++}$.

Algorithm {\tt BCFXY}$^*$ works precisely as  {\tt BCFXY} \cite{cunha:2019}, which is based on the similar formulation  $\mathcal{F}_{xy}$. Both algorithms implement the same cutting plane engine, that separates only one class of valid inequalities, SECs \eqref{eq:sec}. 

There are two minor differences between the two methods, {\tt BCFXY} and {\tt BCFXY}$^*$. First, the latter does not include inequalities \eqref{eq:couplexy1} in the LPRs. Since variables $\mathbf{y}$ are not enforced to be integer for the {\tt BCFXY}$^*$ case,  the strong branching approach implemented by the MIP solver does not investigate their impact on branching, possibly saving some CPU time. Besides that, another possible advantage of not branching on $\mathbf{y}$ has to do with the balancedness of the branch-and-bound search tree. The search tree tends to be more balanced when branching is carried out on $\mathbf{x}$ first. Our reasoning is the following. Suppose {\tt BCFXY} branches on $y_{ij}$, creating two nodes: ``$y_{ij} = 1$" and ``$y_{ij} = 0$". The branch-and-bound node corresponding to ``$y_{ij} = 1$" is more restricted than the node ``$x_{ij} = 1$" that would be obtained if the algorithm had branched on $x_{ij}$. That applies because $y_{ij} = 1 \rightarrow x_{ij} = 1$ and $x_{ik} = 0, \vec{ik} \in R_i: \measuredangle_{0ik}^j > \measuredangle_{0ij} + \alpha$. However, the other branch-and-bound node, corresponding to ``$y_{ij} = 0$" is less restricted than the corresponding ``$x_{ij} = 0$" case, since $y_{ij} = 0$ does not necessarily imply $x_{ij} = 0$. Thus, a less balanced search tree likely results if the algorithm systematically (and unnecessarily) branches first on $\mathbf{y}$. 

Due to the similarities between {\tt BCFXY}$^*$ and {\tt BCFXY},  this section concentrates more on the description of algorithm {\tt BCFX}$^{++}$, which brings important new ingredients to  {\tt BCFX}. In addition to SECs \eqref{eq:sec}, the cutting plane engine in {\tt BCFX}$^{++}$ separates two other classes of valid inequalities: LACs \eqref{eq:strenthen3} and odd-cycles \eqref{eq:cycleineq}. 

{\tt BCFX}$^{++}$ starts off solving the following relaxation for $\mathcal{F}_x^{++}$

\begin{equation}
\label{lpfx}
\min \left \{ \sum_{e \in E} w_e x_e: \mathbf{x} \in \overline{{\cal F}}_{x}\right \},
\end{equation}
\noindent where polyhedral region $\overline{{\cal F}}_{x} \supset {\cal F}_{x}^{++}$ denotes the intersection of constraints \eqref{eq:cardtree} and  \eqref{eq:xbound} and $\mathbf{x} \leq \mathbf{1}_m$, where $\mathbf{1}_m$ denotes the $m-$dimensional vector of ones. SECs \eqref{eq:sec}, LACs  \eqref{eq:strenthen3} and odd-cycles \eqref{eq:cycleineq} are not included in the initial relaxation $\overline{{\cal F}}_{x}$ and added to the relaxation on the fly. In what follows, denote by $\overline{\mathbf{x}} \in [0,1]^m$ an optimal solution to \eqref{lpfx}. 

The cutting plane algorithm embedded in {\tt BCFX}$^{++}$  separates SECs,  LACs  \eqref{eq:strenthen3} and odd-cycles \eqref{eq:cycleineq}, according to this order.  One class of valid inequalities is only separated if $\overline{\mathbf{x}}$ violates no inequality of the preceding classes. Violated inequalities are appended to $\overline{{\cal F}}_x$, thus resulting in an updated reinforced relaxation $\overline{{\cal F}}_x$ to ${\cal F}_x^{++}$, and the algorithm then iterates. Better computational results were found when the separation of odd cycles was restricted to branch-and-bound nodes with depth equal or smaller than three (the root node having depth equal to one). The other two families of valid inequalities are always separated for all branch-and-bound nodes.

{\tt BCFX}$^{++}$ separates SECs precisely as {\tt BCFX} does, combining the separation heuristics of \citet{bicalho:2016} and the exact algorithm introduced by \citet{padberg:1983}. In fact, the same SEC separation procedures are shared by algorithms {\tt BCFXY}, {\tt BCFXY}$^*$,  {\tt BCFX} and {\tt BCFX}$^{++}$. An important feature of the SEC separation strategy is that the exact separation is only called if the heuristic fails on finding a SEC violated by $\overline{\mathbf{x}}$. Details on how the two procedures work can be found in \cite{padberg:1983,bicalho:2016,cunha:2019}.

LACs \eqref{eq:strenthen3} are separated by the following exact algorithm, called for every $i \in V$. Define $\overline{\delta}(i) := \{ e \in \delta(i): \overline{x}_e > 0\}$. The algorithm enumerates all possible subsets of edges in $\overline{\delta}(i)$. For each subset $s(i) \subseteq \overline{\delta}(i)$, it computes the amount $v^{(s(i),\alpha)}$ and checks whether $\sum_{e \in s(i)}\overline{x}_e > v^{(s(i),\alpha)}$ holds. In positive case, a violated inequality \eqref{eq:strengthen1} is found. The algorithm then implements the sequential lifting procedure discussed in Section \ref{sec:lifting}, so that an inequality \eqref{eq:strengthen1} is lifted to the stronger form \eqref{eq:strenthen3}. All violated inequalities are stored in a list; one list dedicated to each $i \in V$. Only the most violated inequality in each list is added to the LPR relaxation $\overline{{\cal F}}_x$. Therefore, at most $n$ violated inequalities are appended to the new relaxation, whenever the enumeration is called.  Although the algorithm outlined above runs in exponential time, it is very fast in practice, because $\overline{\delta}(i)$ includes just few edges.

The procedure for separating odd-cycle inequalities we implemented is also exact, in the sense that it always finds a violated inequality by $\overline{\mathbf{x}}$, provided that one exists. The algorithm is described in detail in \cite[Section 4.1]{rebennack:2012}. In short, it receives $\overline{\mathbf{x}}$ and $G_c$ as input and returns the minimum weight odd cycle of $G_c$. In addition, it checks if that cycle is chordless. If it is not, it identifies the odd holes in that cycle and returns them. The algorithm creates an auxiliary weighted bipartite graph $\hat{G} = (\hat{V}_c,\hat{E}_c)$ from $G_c$ as follows. The vertex set $\hat{V}_c$ involves two copies $e^+,e^-$ of each vertex $e \in V_c$. Vertices labelled `+' define one set of the partition; the other is defined by vertices labelled `-'. The edge set is defined as $\hat{E}_c = \{  \{f^+,e^-\}, \{e^+,f^-\} : \{e,f\} \in E_c\}$.  The weight of an edge $\{e^+,f^-\} \in \hat{E}_c$ is  $\frac{1-\overline{x}_e - \overline{x}_f}{2}$, where $f$ and $e$ respectively represent the edges of $E$ corresponding to the vertices of $G_c$. Then, for every $e \in V_c$, the algorithm computes the shortest path of $\hat{G}_c$ that connects $e^+$ to $e^-$. From that path in $\hat{G}_c$, the algorithm extracts an odd cycle of $G_c$ and checks whether or not the cycle is chordless. In case it is not, an odd hole is retrieved from the odd cycle.

Since the odd cycle separation algorithm is only called when no LACs are violated, all constraints of the type $x_e + x_f \leq 1$ are always satisfied by $\overline{\mathbf{x}}$, for every edge $\{e,f\}$ of the conflict graph. Therefore, the weights $\frac{1-\overline{x}_e - \overline{x}_f}{2}$ are non-negative and the shortest path computations can be carried out by Dijkstra's algorithm. The algorithm's complexity is dominated by the shortest path computations. Thus, using Dijskstra's algorithm, it runs in $O(|V_c||E_c|\log(|V_c|))$ time.  In order to improve its practical performance, the shortest path from $e^+$ to $e^-$ is only called if, at that separation round, the edge $e$ of $E$ was not already included in a violated cycle inequality. Not only such an strategy reduces the number of calls of Dijkstra's algorithm, but also tends to generate  {\it sufficiently orthogonal} violated inequalities. Following such an strategy, we managed to include all violated inequalities in the new relaxation $\overline{{\cal F}}_{x}$, without excessive impact on linear programming reoptimization cost.

Before the very first relaxation \eqref{lpfx} is solved, the Kruskal-like heuristic introduced in \cite{cunha:2019} is called, to provide valid $\alpha-$MSTP upper bounds for 
 {\tt BCFX}$^{++}$. The heuristic is called again, at the end of each branch-and-bound node. However, instead of using the original edge weights $\{w_e: e \in E\}$ as input, the heuristic is called under weights modified by the optimal solution $\overline{\mathbf{x}}$ to the last LPR solved at that node. More precisely, the heuristic is called under modified costs $\{w_e(1 - \overline{x}_e): e \in E\}$. The same strategy is used for all the other $\alpha-$MSTP BC algorithms discussed in this paper. 
 
Finally, {\tt BCFX}$^{++}$ is implemented under the {\tt XPRESS}  MIP package, release 8.4 \cite{xpress-2017}. {\tt XPRESS} is thus responsible for solving LPRs, \eqref{lpfx}, and managing the BC tree. It uses default options to choose a variable to branch on and implements a best-first search strategy. Additional features offered by {\tt XPRESS}, such as automatic cut generation and primal heuristics, are kept switched off. Likewise, multi-threading is not used as well.

%===========================================================================
%Resultados computacionais
%\input{results-v03.tex}
\section{Computational experiments}
\label{sec:results}

In this section, we numerically evaluate the strength of the bounds $w({\cal F}_x^{+})$ and $w({\cal F}_x^{++})$ and compare them to $w({\cal F}_x)$ and $w({\cal F}_{xy}^*$). In addition, we report on the computational experiments with the $\alpha-$MSTP BC algorithms {\tt BCFX}$^{++}$ and {\tt BCFXY}$^{*}$.

The algorithms described here were implemented in {\tt C} and compiled with {\tt gcc}, with optimization flags {\tt -O3} turned on, under the Linux  operating system (release LTS 14.04).
A computer equipped with an Intel XEON E5645 processor, running at 2.4GHz and having 32Gb of RAM memory (12Mb of cache memory), was used for the experiments. Our computational results were generated with the same computational environment used in \cite{cunha:2019}. Additionally, the same MIP solver (and release) was used here and in \cite{cunha:2019}, in order to manage the BC search trees. Therefore, our results and those in \cite{cunha:2019} are directly comparable. 

\subsection{Test instances}

The computational experiments reported here were conducted with the $\alpha-$ MSTP instances suggested in \cite{cunha:2019}. They were generated from two dimensional Euclidean graphs, corresponding to Euclidean Traveling Salesman Problem (ETSP) instances of the TSPLIB \cite{tsplib:1991}. Out of the underlying set of points of the ETSP instance, three distinct graphs $G = (V,E)$ were generated in \cite{cunha:2019}. One of them involves all vertices of the TSPLIB instance, while the two others involve a lesser number of vertices. For example, take TSPLIB instance {\tt berlin52} that has 52 vertices.  Graphs with $n \in \{ 15, 30, 52\}$ vertices were then generated out of the original set of 52 Euclidean plane points. The largest of them, the TSPLIB graph itself. The other two, involving respectively the first 15 and the first 30 points. Each edge set $E$ was always complete, irrespective of the value of $n$. In addition,, for every applicable pair of vertices, their corresponding Euclidean distances, $\{w_e: e=\{i,j\} \in E\}$, were taken as edge weights. Full {\tt double} precision was used for computing these distances. In total, 39 distinct graphs were thus generated. 

Our numerical investigation is dedicated to values of $\alpha$ in the interval $\alpha < \pi$. Reasons for not testing instances with $\alpha \geq \pi$ are twofold. First, the numerical results reported in \cite{cunha:2019} showed that they are considerably easier then their $\alpha < \pi$ counterparts. In addition, the set partitioning structure characterized here, ${\cal I}_1$, is useless for strengthening LPRs for them.  

For each of the 39 graphs generated in \cite{cunha:2019}, we considered 6 values of $\alpha \in [\frac{\pi}{3},\pi)$. In addition to the values already tested in \cite{cunha:2019}, $\alpha \in \{ \frac{\pi}{3},\frac{2\pi}{3}\}$, we considered four new values, namely $\alpha \in \{\frac{2\pi}{5},\frac{\pi}{2}, \frac{3\pi}{5},\frac{4\pi}{5}\}$. As in \cite{cunha:2019}, smaller values of $\alpha$  were not tested here since \citet{aschner:2017} showed that $\alpha-$STs are only guaranteed to exist when $\alpha \geq \frac{\pi}{3}$ applies.

In total, 234 instances, corresponding to 39 graphs for each of the 6 values of $\alpha$, were tested. Due to the large number of test instances, the main text body of the paper presents only aggregated results, which indicate more general trends. Detailed computational results, for each value of $\alpha$ and input graph, are presented in an accompanying supplementary material.

\subsection{Comparison of LPR bounds}
\label{sec:reslprs}

In this section, we numerically evaluate the impact of the $\alpha-$MSTP valid inequalities \eqref{eq:strenthen3} and \eqref{eq:cycleineq}, for strengthening the existing LPR relaxation bounds for the problem. We also take into consideration the computational effort needed to evaluate these bounds.

The first important observation to be made is that, according to Theorem \ref{theoprojection}, $w({\cal F}_{x}^+) \leq w({\cal F}_{xy}^*)$ holds. In practice, for all 234 instances in our test bed, bounds $w({\cal F}_{x}^+)$ matched $w({\cal F}_{xy}^*)$ counterparts. However, we did not manage to prove or to disprove that $w({\cal F}_{x}^+)$ and $w({\cal F}_{xy}^*)$ are always equally strong. %However, f\textcolor{red}{this empirical result combined with  the fact that bounds $w({\cal F}_x)$ are substantially weaker than all other LPR relaxation bounds in the literature provides evidence that  the strengthened inequalities \eqref{eq:strenthen3} are indeed substantially stronger than the previously available valid inequalities \eqref{cover} for modeling $\alpha-$ACs, in terms only of the $\mathbf{x}$ variables.} 

We now discuss how strong bounds $w({\cal F}_x^{++})$ are compared to $w({\cal F}_x)$ and $w({\cal F}_{xy}^*)$. In Table \ref{tab:lpsavg}, we report the average gaps $\frac{w({\cal F}_x^{++})-w({\cal F}_x)}{w({\cal F}_x)}$ and $\frac{w({\cal F}_x^{++})-w({\cal F}_{xy}^*)}{w({\cal F}_{xy}^*)}$, in percentage values. For each value of $\alpha$, the table reports values averaged over the 39 graphs. The table also presents $\frac{t({\cal F}_x^{++})}{t({\cal F}_x)}$, the ratio between the average CPU times needed to compute bounds $w({\cal F}_x^{++})$ and $w({\cal F}_x)$, as well as $\frac{t({\cal F}_x^{++})}{t({\cal F}_{xy}^*)}$, similarly defined. 

%===========================================================================
%Tabela Raiz
%\input{./Tabelas/tabrootcpu.tex}
\begin{table}[!b]
\begin{center}
{\footnotesize
\caption{Summary of LPR bounds and LPR CPU time ratios.}
\label{tab:lpsavg}       
\renewcommand{\arraystretch}{1.4} 
\renewcommand{\tabcolsep}{0.8mm}  
\begin{tabular}{crrrrrr}
\hline
$\alpha$	&	& \multicolumn{2}{c}{LPR bounds} & & \multicolumn{2}{c}{Ratios of CPU times to} \\
                &	& \multicolumn{2}{c}{quality (\%)} & & \multicolumn{2}{c}{compute LPR bounds} \\
\cline{3-4} \cline{6-7}
                 &     & $\frac{w({\cal F}_x^{++})-w({\cal F}_x)}{w({\cal F}_x)}$ & $\frac{w({\cal F}_x^{++})-w({\cal F}_{xy}^*)}{w({\cal F}_{xy}^*)}$ & & $\frac{t({\cal F}_x^{++})}{t({\cal F}_x)}$ & $\frac{t({\cal F}_x^{++})}{t({\cal F}_{xy}^*)}$ \\
$\frac{\pi}{3}$ & & 13.89 & 1.60 & & 59.63 & 1.50 \\	
$\frac{2\pi}{5}$ & & 9.39 & 2.02 & & 27.78 & 0.87 \\	
$\frac{\pi}{2}$ & & 4.73 & 1.85 & & 13.89 & 0.51 \\	
$\frac{3\pi}{5}$ & & 1.75 & 1.23 & & 8.62& 0.37 \\
$\frac{2\pi}{3}$ & & 0.82 & 0.76 & & 6.34 & 0.26 \\
$\frac{4\pi}{5}$ & & 0.13 & 0.13 & & 4.67 & 0.16 \\		
\hline
\end{tabular}
}
\end{center}
\end{table}
%===========================================================================

Computational results reported in Table \ref{tab:lpsavg} show that, for small values of $\alpha$, formulation ${\cal F}_{x}$ is substantially weaker than ${\cal F}_{x}^{++}$. However, bounds $w({\cal F}_{x})$ are much cheaper to be evaluated than  $w({\cal F}_{x}^{++})$. They also suggest that formulation ${\cal F}_{xy}^*$ is, on average, around 2\% weaker than ${\cal F}_{x}^{++}$, for the three smallest values of $\alpha$. Since  bounds $w({\cal F}_{x}^+)$ and $w({\cal F}_{xy}^*)$ are identical for all instances in our test bed, the fact that bounds $w({\cal F}_{x}^{++})$ exceed the best LPR relaxation bounds introduced in \cite{cunha:2019}, $w({\cal F}_{xy})$, comes exclusively from the use of cycle inequalities \eqref{eq:cycleineq}. Notice that, except for the $\alpha = \frac{\pi}{3}$ case, not only the bounds $w({\cal F}_{xy}^*)$ are weaker than $w({\cal F}_{x}^{++})$ but the CPU times needed to evaluate them are also larger than those needed to evaluate $w({\cal F}_{x}^{++})$. One of the reasons for that, already highlighted by \citet{cunha:2019}, is that inequalities \eqref{eq:allowed1} become dense as $\alpha$ grows. The fact that LPR bounds $w({\cal F}_{xy}^*)$ are expensive to be evaluated, compared to $w({\cal F}_{x}^{+})$ and $w({\cal F}_{x}^{++})$, refrained us from separating inequalities \eqref{eq:cycleineq}, within {\tt BCFXY}$^{*}$.

%Such results suggest that a BC algorithm based on $w({\cal F}_{x}^{++})$ should outperform the other BC algorithms based on $w({\cal F}_{xy}^*)$. This is the subject of our next section.

\subsection{Computational results for the BC algorithms}
\label{sec:resbcs}

In this section, we compare four BC algorithms: {\tt BCFX} and {\tt BCFXY}, from the literature \cite{cunha:2019}, {\tt BCFX}$^{++}$ and {\tt BCFXY}$^*$, introduced here. Each algorithm was allowed to run for a time limit of 2 CPU hours, for each value of $\alpha$ and input graph.

Table \ref{tab:optcert} reports, for each value of $\alpha$,   the number (out of 39) of optimality certificates  obtained by each algorithm. According to these results,  {\tt BCFX}$^{++}$ is capable of solving more instances to proven optimality than its competitors, within the imposed time limit, for the entire spectrum of $\alpha$ values tested here. For the smallest values of $\alpha$, {\tt BCFX} is not competitive with the best algorithm in \cite{cunha:2019}, {\tt BCFXY}, in terms of the number of optimality certificates. However, due to the use of valid inequalities \eqref{eq:strenthen3} and \eqref{eq:cycleineq}, algorithm {\tt BCFX}$^{++}$, the enhanced version of {\tt BCFX} that also relies on a formulation defined exclusively on the natural space of variables $\mathbf{x}$, outperformed {\tt BCFXY} (and {\tt BCFXY}$^*$) in that respect.

We complement the evaluation of the impact of inequalities \eqref{eq:strenthen3} and \eqref{eq:cycleineq} by presenting, in Table \ref{tab:fxs}, a direct comparison between {\tt BCFX}$^{++}$ and {\tt BCFX}. The columns of the table are split in two blocks. The first one is dedicated to those instances both algorithms managed to solve, within the imposed time limit. The following information are provided, for that block: the number of instances solved to proven optimality by {\tt BCFX} and {\tt BCFX}$^{++}$, the average CPU time (in seconds) taken by each algorithm to solve these instances, and the number of times each algorithm was the fastest of the two. The second block of columns addresses those instances no algorithm managed to solve within the time limit. The table presents the number of instances that could not be solved by both methods, followed by the number of times each algorithm delivered the strongest best upper bounds (BUB), when the time limit was hit. Results given in Table \ref{tab:fxs} show that, on the average, {\tt BCFX}$^{++}$ was faster than {\tt BCFX} for the entire range of $\alpha$ values. Considering the hardest cases, i.e., those defined for $\alpha \in \{\frac{\pi}{3}, \frac{2\pi}{5},\frac{\pi}{2}\}$, {\tt BCFX}$^{++}$ is at least 2 orders of magnitude faster than {\tt BCFX}, on the average.  For the largest values of $\alpha$ tested in our study, {\tt BCFX}$^{++}$ was faster than {\tt BCFX} in fewer cases.  {\tt BCFX}$^{++}$ clearly outperforms {\tt BCFX} for those instances that were not solved by both. Often, root node lower bounds $w({\cal F}_{x}^{++})$ computed by {\tt BCFX}$^{++}$ are  stronger than the best globally valid lower bounds computed by {\tt BCFX}, after investigating thousands of nodes at the end of the time limit. Furthermore, {\tt BCFX}$^{++}$ provided the $\alpha-$STs with the lowest costs, when the time limit was hit. 

As a general observation, the impact of inequalities \eqref{eq:strenthen3} and \eqref{eq:cycleineq} decreases as $\alpha$ gets close to $\pi$. In fact, the lifting \eqref{eq:strenthen3} of inequalities \eqref{cover} becomes less important, since for more non-admissible sets $s(i)$ the value of $v^{(s(i),\alpha)}$ do not change from $|s(i)|-1$. In addition, the density of the conflict graphs $G_c$ becomes smaller, since fewer pairs of edges are non-admissible. Thus, fewer valid inequalities \eqref{eq:cycleineq} are expected to be characterized.

A direct comparison between {\tt BCFXY}$^*$ and {\tt BCFXY} is presented in Table \ref{tab:fxys}. In addition to the type of information provided earlier in Table \ref{tab:fxs}, Table \ref{tab:fxys} also provides the average number of branch-and-bound nodes investigated for the instances solved by both algorithms. Considering those instances, {\tt BCFXY}$^*$ is, on the average, faster than {\tt BCFXY} for all values of $\alpha$, except for $\alpha = \frac{4}{5}\pi$, the easiest value of $\alpha$ considered here. In many more instances, {\tt BCFXY}$^*$ attained the smallest CPU times. Except for the two extreme values of $\alpha$ tested here, $\alpha = \frac{\pi}{3}$ and $\alpha = \frac{4\pi}{5}$, fewer nodes are explored by {\tt BCFXY}$^*$ on the average, for those instances solved by both methods. In part, such results confirm our claim that, in general, more balanced (and possibly smaller) branch-and-bound search trees should result when one avoids branching on $\mathbf{y}$.
Moving our focus now to those instances left unsolved by both, the advantage of {\tt BCFXY}$^*$ over {\tt BCFXY}, is not so pronounced as far the quality of the best upper bound at the end of the imposed time limit is concerned. In general, {\tt BCFXY}$^*$ delivers better feasible solutions than {\tt BCFXY} at the end of the time limit, for the largest values of $\alpha$, while the opposite holds for the two smallest values of $\alpha$. 

The results discussed above suggest that algorithm {\tt BCFX} is not on pair with the other three algorithms compared here, specially for the hardest instances. Thus, we restrict the comparison presented at Table \ref{tab:bcs} to the other three methods.  The table presents information following previously explained pattern of data. However, its second block also gives the number of times each of the three algorithms delivered the strongest globally valid lower bounds (BLB). 

Our discussion of the results given in Table \ref{tab:bcs} is divided in two parts: one for $\alpha = \frac{\pi}{3}$ and another for $\alpha \geq \frac{2\pi}{5}$. For the latter,  {\tt BCFX}$^{++}$ provided the best results. Considering the instances solved by the three, on average, it was never the slowest. On the contrary, except for the $\alpha = \frac{\pi}{2}$ case, it was the fastest on the average. Furthermore, in more cases for all values of $\alpha$ it was the fastest of the three. Now focusing on the instances left unsolved, in more cases {\tt BCFX}$^{++}$ delivered not only the strongest lower bounds at the end of the search (as one would expect since it is based on the strongest formulation), but it also provided the sharpest upper bounds. 

    We now compare the three methods for the $\alpha = \frac{\pi}{3}$ instances. {\tt BCFXY} obtained the smallest number of optimality certificates out of the 39 available instances. {\tt BCFX}$^{++}$ solved all the 14 instances solved by {\tt BCFXY} plus two others.  
In fact, {\tt BCFX}$^{++}$ stands out with the largest success rates, measured by the number of optimality certificates. Despite the fact that {\tt BCFXY} has the best average CPU times, it is hardly the fastest of the three, since in just one case it provided the smallest CPU times for the 14 instances solved by the three methods. While {\tt BCFX}$^{++}$ was the fastest in 4 out of these 14 instances, the best rate was attained by {\tt BCFXY}$^{*}$, the fastest in 9 cases.
As for the other values of $\alpha$ tested here, {\tt BCFX}$^{++}$ also provides the best lower bounds at the end of the time limit, for the instances no algorithm could solve. Overall, the dominance of {\tt BCFX}$^{++}$ over its competitors is not as evident as for the $\alpha \geq \frac{2\pi}{5}$ instances. Note that, for the 14 instances solved by the three methods, {\tt BCFX}$^{++}$ was the slowest, on the average. Its poor average computational results is mostly explained by the fact that it took 6245.8 seconds to solve instance {\tt pr76} with $n = 50$, while {\tt BCFXY} and {\tt BCFXY}$^{*}$ respectively took only 563.3 and 422.9 seconds to accomplish similar task. Finally, {\tt BCFXY} obtained the best upper bounds in more cases than {\tt BCFXY}$^*$, when the time limit was achieved without obtaining an optimality certificate. To summarize, for $\alpha = \frac{\pi}{3}$ it does not seem to exist a clear winner among the three methods. None of the three methods seems to dominate the others, but slightly inferior results seem to be obtained by {\tt BCFXY}.

Finally, detailed computational results reported in the online supplement to this paper show that {\tt BCFX}$^{++}$ provided 8 new optimality certificates for instances tested in \cite{cunha:2019}: 2 for  the $\alpha = \frac{\pi}{3}$ instances and 6 for the $\alpha = \frac{2\pi}{3}$ ones.

%=================================================================
%Tabela Optimality certificates
%\input{./Tabelas/optcertificates.tex}
\begin{table}
\begin{center}
{\footnotesize
\caption{BC resuls: Number of optimality certificates.}
\label{tab:optcert}       
\renewcommand{\arraystretch}{1.5} 
\renewcommand{\tabcolsep}{0.8mm}  
\begin{tabular}{lrrrrr}
\hline
$\alpha$	    &		&	{\tt BCFX}	&	{\tt BCFXY}	& {\tt BCFXY}$^*$	&	{\tt BCFX}$^{++}$ \\
\cline{2-6}
$\frac{\pi}{3}$ &    & 6              &    14           &    15               &   16\\
$\frac{2\pi}{5}$ &    & 9             &    14           &    17               &   17 \\
$\frac{\pi}{2}$ &    &   15           &    17           &    17               &   18 \\
$\frac{3\pi}{5}$ &    &    21          &   22            &    23               &  24 \\
$\frac{2\pi}{3}$ &    &    30          &   26            &    26               &  30 \\
$\frac{4\pi}{5}$ &    &    39          &   39            &    39               &  39 \\
\hline
Total            &    &   120             &      132         &   137                &  144 \\
\hline
\end{tabular}
}
\end{center}
\end{table}

%=================================================================
%TCompara bcfxs
%\input{./Tabelas/comparabcfxs_apenas.tex}

\begin{table}
\begin{center}
{\footnotesize
\caption{Direct comparison of {\tt BCFX} and {\tt BCFX}$^{++}$.}
\label{tab:fxs}       
\renewcommand{\arraystretch}{1.5} 
\renewcommand{\tabcolsep}{0.8mm}  
\begin{tabular}{lrcrrrrrrcrrr}
\hline
&		&	\multicolumn{6}{c}{Instances}	& &	\multicolumn{4}{c}{Instances left unsolved} \\
	    &		&	\multicolumn{6}{c}{Solved by both}	& &	\multicolumn{4}{c}{by both} \\
	    \cline{2-2}
\cline{3-8} \cline{10-13}
                &       &  \# of cases & \multicolumn{2}{c}{Avg CPU time} & & \multicolumn{2}{c}{\# of times} & & \# of cases & & \multicolumn{2}{c}{obtained better}\\ \cline{3-3} \cline{10-10}
 &       &   & \multicolumn{2}{c}{$t(s)$} & & \multicolumn{2}{c}{was faster} & & & & \multicolumn{2}{c}{BUBs} \\               
\cline{4-5} \cline{7-8}  \cline{12-13}
 $\alpha$               &       &              & {\tt BCFX} & {\tt BCFX}$^{++}$ & & {\tt BCFX} & {\tt BCFX}$^{++}$ & & & & {\tt BCFX} & {\tt BCFX}$^{++}$ \\
$\frac{\pi}{3}$  &    & 6  & 761.0 & 5.7 %20.9  
& & 0 & 6  & & 23 & & 6 & 17  \\
$\frac{2\pi}{5}$ &    & 9  & 246.7 & 4.7  & & 0 & 9  & & 22 & & 3   & 19  \\            
$\frac{\pi}{2}$  &    &15  & 101.3 & 20.1  & & 2 & 13  & & 21 & & 3   & 18  \\       
$\frac{3\pi}{5}$ &    &21  & 326.9 & 121.8 & & 11 & 10 & & 15 & & 2   & 13    \\        
$\frac{2\pi}{3}$ &    &30  & 423.2 & 309.2 & & 25 & 5 & & 9  & & 2   & 7  \\        
$\frac{4\pi}{5}$ &    &39  & 178.9 & 124.5 & & 29 & 10 & & -  & & -   & -   \\           
\hline
\end{tabular}
}
\end{center}
\end{table}

%=================================================================
% Compara bcfxys
%\input{./Tabelas/REVcomparabcfxys_apenas.tex}
\begin{table}
\begin{center}
%{\footnotesize
\caption{Direct comparison of {\tt BCFXY} and {\tt BCFXY}$^{*}$.}
\label{tab:fxys}       
\renewcommand{\arraystretch}{1.5} 
\renewcommand{\tabcolsep}{0.8mm}  
\begin{tabular}{lrcrrrrrrrrrcrrr}
\hline
&		&	\multicolumn{9}{c}{Instances}	& &	\multicolumn{4}{c}{Instances left unsolved} \\
	    &		&	\multicolumn{9}{c}{Solved by both}	& &	\multicolumn{4}{c}{by both} \\
	    \cline{2-2}
\cline{3-11} \cline{13-16}
                &       &  \# of cases & \multicolumn{2}{c}{Avg CPU time} & & \multicolumn{2}{c}{\# of times} & & \multicolumn{2}{c}{Avg \# of} & & \# of cases & & \multicolumn{2}{c}{obtained better}\\ \cline{3-3} \cline{13-13}
 &       &   & \multicolumn{2}{c}{$t(s)$} & & \multicolumn{2}{c}{was faster} & & \multicolumn{2}{c}{of nodes} & & & & \multicolumn{2}{c}{BUBs} \\               
\cline{4-5} \cline{7-8} \cline{10-11} \cline{14-16}
 $\alpha$               &       &              & {\tt BCFXY} & {\tt BCFXY}$^{*}$ & & {\tt BCFXY} & {\tt BCFXY}$^{*}$ & & {\tt BCFXY} & {\tt BCFXY}$^{*}$ & & & & {\tt BCFXY} & {\tt BCFXY}$^{*}$ \\
$\frac{\pi}{3}$  &    & 14  & 241.6 &  274.4 &  & 2 & 12  & & 12969 & 14697 & &  23 &  &  12& 11  \\
$\frac{2\pi}{5}$ &    & 17  & 342.9 &  200.3 & & 0 & 17  & & 15900 & 14207 & & 21 & &   12 & 9  \\            
$\frac{\pi}{2}$  &    & 16 & 53.5 & 33.7       & &  1 &  15 & &  4312 & 3548    & & 21 & &    8 & 13  \\       
$\frac{3\pi}{5}$ &    & 22 & 681.2&  536.5  & &  3 & 19 & &  10374 & 9415 & & 16 & &    7 &  9   \\        
$\frac{2\pi}{3}$ &    & 25 & 305.9 & 180.2 & &  2&  23 & &  7809 & 3933   & & 12 & &    3 &   9 \\        
$\frac{4\pi}{5}$ &    & 39  &  264.7        &  280.1         & &  5&  34 & & 1285 & 2138 & & -  & & -   & -   \\           
\hline
\end{tabular}
%\renewcommand{\arraystretch}{1}
%}
\end{center}
\end{table}

%=================================================================
% Compara todos bcs
%\input{./Tabelas/compara_3bcs.tex}
\begin{table}
%\begin{center}
{\scriptsize
\caption{Direct comparison of {\tt BCFXY}, {\tt BCFXY}$^{*}$ and {\tt BCFX}$^{++}$.}
\label{tab:bcs}       
\renewcommand{\arraystretch}{1.4} 
\renewcommand{\tabcolsep}{0.5mm}  
\begin{tabular}{lrcrrrrrrccccccccccc}
\hline
&		&	\multicolumn{8}{c}{Instances solved}	& &	\multicolumn{9}{c}{Instances left unsolved} \\
	    &		&	\multicolumn{8}{c}{by the three}	& &	\multicolumn{9}{c}{by the three} \\
	    \cline{2-2}
\cline{3-10} \cline{12-20}
                &       &  \# of & \multicolumn{3}{c}{Avg CPU time} & & \multicolumn{3}{c}{\# of times} & & \# of & &  \multicolumn{7}{c}{obtained the best}\\ \cline{14-20}
 &       & cases   & \multicolumn{3}{c}{$t(s)$} & & \multicolumn{3}{c}{was faster} & & cases & & \multicolumn{3}{c}{BUBs} & &  \multicolumn{3}{c}{BLBs}\\               
\cline{4-6} \cline{8-10}  \cline{12-12} \cline{14-16} \cline{18-20}

 $\alpha$               &       &              & {\tt BCFXY} & {\tt BCFXY}$^{*}$ & {\tt BCFX}$^{++}$ & & {\tt BCFXY} & {\tt BCFXY}$^{*}$ & {\tt BCFX}$^{++}$ &  & & & {\tt BCFXY} & {\tt BCFXY}$^{*}$ & {\tt BCFX}$^{++}$ & & {\tt BCFXY} & {\tt BCFXY}$^{*}$ & {\tt BCFX}$^{++}$ \\
$\frac{\pi}{3}$  &    & 14  & 241.6 & 274.4  & 764.1 &  &  1 & 9 &  4  &  &  23 & & 10 & 6 & 7  & & 3 & 4 & 16 \\
$\frac{2\pi}{5}$ &   &  16 & 267.3 & 122.2  & 111.9 &  &  0  & 5 & 10 &  &  21 &  & 6 & 7 & 8 & & 0 & 1 & 20  \\          
$\frac{\pi}{2}$  &    & 16  & 53.5 &  33.7  & 41.6 &   &   0 &  3 & 12 &  &  20 & & 5 & 9 & 6 & & 0 & 0 & 20 \\
$\frac{3\pi}{5}$ &   &  22 & 681.2 &  538.5 & 219.9 &  &  0 &  2  & 20  &  & 15 &   & 0 & 1 & 14 & & 0 & 0 & 15 \\      
$\frac{2\pi}{3}$ &   &  24 & 258.3  & 168.1  & 42.2 &  &  0  & 0 &  24  &  & 7  & & 0 & 2 & 5 & & 0 & 0 & 7  \\  
$\frac{4\pi}{5}$ &   &  39 &  264.7 & 280.0  & 124.5 &  &  1  & 0 &  36  &  & -  & & - & - & - & & - & - & -  \\             
\hline
\end{tabular}
}
%\end{center}
\end{table}

%===========================================================================
% Conclusoes
%\input{conclusions-v02.tex}
\section{Conclusions}
\label{sec:conclusions}

In this study, we presented two improved formulations for $\alpha-$MSTP, $\mathcal{F}_{xy}^*$ and $\mathcal{F}_x^{++}$. Despite the fact that $\mathcal{F}_{xy}^*$ differs by minor aspects from $\mathcal{F}_{xy}$, an equally strong formulation coming from the literature \cite{cunha:2019},  the BC algorithm based on $\mathcal{F}_{xy}^*$ seems to outperform the one based on $\mathcal{F}_{xy}$. Most likely the reasons being the smaller CPU times involved in linear programming reoptimization.

The second formulation introduced here, $\mathcal{F}_{x}^{++}$, is significantly stronger than $\mathcal{F}_{x}$, another formulation also introduced in \cite{cunha:2019}. It uses a much stronger family of valid inequalities to enforce the required angular constraints. Additionally, it also incorporates valid inequalities for the Stable Set polytope, whose structure in $\alpha-$STs defined by values of $\alpha$ in the interval $(0,\pi)$ was disclosed here. In fact, formulation $\mathcal{F}_{x}^{++}$ is stronger than all the other formulations in the literature. Thanks to that, for the majority of the instances tested here, its accompanying Branch-and-cut algorithm, {\tt BCFXY}$^{++}$, obtained the best computational results and seems to be the best $\alpha-$MSTP available  exact algorithm.

As for future research, we plan to further investigate valid inequalities for $\alpha-$MSTP. In particular, we have not explored valid inequalities for conflict hypergraphs associated to $\mathcal{I}$.  So far, little effort has been dedicated to the development of heuristics and meta-heuristics for the problem. We believe these may be interesting avenues for future investigation.

%===========================================================================
% apendice
% \input{appendix-v02.tex}
\appendix
\label{sec:app}

\section{Proof of Theorem \ref{theoprojection}}
\label{sec:apptheo}

The goal of this appendix is to prove Theorem \ref{theoprojection}, i.e., that $\hbox{Proj}_x(\mathcal{F}_{xy}^*) \subseteq \mathcal{F}_x^+$.  To proceed with the proof, we need the following auxiliary result.

\begin{lemma}
\label{lemnovo}
Given any $s(i) \subset \delta(i), s(i) \not = \emptyset$ and $\{i,u\} \in \delta(i) \setminus s(i)$, it holds:
\begin{equation*}
v_{iu}^{(s(i),\alpha)} \leq v^{(s(i),\alpha)}.
\end{equation*}
\end{lemma}
\begin{proof}
For the proof, we say that an edge $\{i,j\}$ covers $\{i,z\}$ if $j \in L_{iz}$. Denote by $Z_i = \{ \vec{ij} \in R_i: \{i,j\} \in s(i) \}$. We have the following cases to consider:
\begin{enumerate}
\item $\vec{iu}$ is collinear to $\vec{ij} \in Z_i$. \\ Then, the edges $\{i,k\} \in s(i)$ that are covered by $\{i,u\}$ are precisely those covered by $\{i,j\}$. Then, we have:  
$v_{iu}^{(s(i),\alpha)}=v_{ij}^{(s(i),\alpha)} = |\{ \{i,k\} \in s(i): \measuredangle_{jik} \leq \alpha \} | \leq v^{(s(i),\alpha)}$, where the last inequality follows from the definition of $v^{(s(i),\alpha)}$.
\item $\vec{iu}$ is colinear to no vector $\vec{ij} \in Z_i$. \\
Now, it suffices to find any edge $\{i,z\} \in s(i)$ such that $v_{iu}^{(s(i),\alpha)} \leq v_{iz}^{(s(i),\alpha)}$. To that aim, 
assume that $s(i) = \{ \{i,v_1\},\dots,\{i,v_p\} \}$, where $p = |s(i)|$ and the vectors in $Z_i$ are indexed so that the ordering \eqref{ordering} holds. Let $\{i,z\}$ be the edge corresponding to the vector $\vec{iz}$ such that:
$$\vec{iz} = \arg\min \{ \measuredangle_{uiv_k}:  \vec{iv_k} \in Z_i \}.$$
It is clear that all edges in $s(i) \setminus \{i,z\}$ that are not covered by $\{i,z\}$ are also not covered by $\{i,u\}$. Since $\{i,z\}$ covers itself, and $\{i,u\}$ may or not cover $\{i,z\}$ we have that $v_{iu}^{(s(i),\alpha)} \leq v_{iz}^{(s(i)\setminus \{i,z\},\alpha)} + 1 = v_{iz}^{(s(i),\alpha)}\leq v^{(s(i),\alpha)}$.
\end{enumerate}
\end{proof}

\begin{proof}[Proof of Theorem \ref{theoprojection}]
Now, to proceed with the proof of Theorem \ref{theoprojection}, we provide an explicitly description of $\hbox{Proj}_x({{\cal F}_{xy}^*})$, by means of projection cuts. We also address the separation problem: Given $\overline{\mathbf{x}} \in [0,1]^{m}$, decide whether or not $\overline{\mathbf{x}} \in \hbox{Proj}_x({{\cal F}_{xy}^*})$ applies. In case $\overline{\mathbf{x}}  \not \in \hbox{Proj}_x({{\cal F}_{xy}^*})$, then identify an inequality that is satisfied by any point in  $\hbox{Proj}_x({{\cal F}_{xy}^*})$ that is violated by $\overline{\mathbf{x}}$. To that aim, define the following Farkas multipliers: 
\begin{itemize}
\item $\tau_i \in \mathbb{R}: i \in V$, associated to \eqref{eq:assigny} and 
\item $\beta_{ij}^i \geq 0, \: i \in V, \{i,j\} \in \delta(i)$, associated to \eqref{eq:allowed1}. 
\end{itemize} 

Projecting out variables $\mathbf{y}$ from the system of inequalities \eqref{eq:assigny} and \eqref{eq:allowed1}, we obtain the following family of (aggregated) projection cuts
\begin{equation}
\label{eq:projcuts01}
\sum_{i \in V}\sum_{e  = \{i,j\} \in \delta(i)}\beta_{ij}^ix_e \leq \sum_{i \in V} \tau_i,
\end{equation}
\noindent for all vectors of Farkas multipliers satisfying \eqref{proj1} and \eqref{proj2}.
\begin{align}
\sum_{k| \measuredangle_{jik} \leq \alpha}\beta_{ik}^i & \leq \tau_i 
 & & i \in V, \: \: \vec{ij} \in R_i \label{proj1} \\
 \beta_{ij}^i & \geq 0 & & i \in V, \: \: \vec{ij} \in R_i. \label{proj2}
\end{align}
 Thus, $\hbox{Proj}_x({{\cal F}_{xy}^*})$ is defined as $\hbox{Proj}_x({{\cal F}_{xy}^*}) =$ $\left \{ \mathbf{x} \in {\cal F}_{{\cal T}}: \mathbf{x} \hbox{ satisfies }  
\eqref{eq:projcuts01} \hbox{ for all } \right.$ $\left. \hbox{multipliers that satisfy } \eqref{proj1}-\eqref{proj2} \right \}.$
%Now set to zero all Farkas multipliers associated to $i \in V \setminus S$, for any given $S \subset V$. The corresponding projection cuts then read as
%\begin{equation}
%\label{eq:projcuts02}
%\sum_{i \in S} \sum_{e  = \{i,j\} \in \delta(i)}\beta_{ij}^ix_e    \leq \sum_{i \in S}\tau_i.
%\end{equation}
%\noindent  

Inequalities \eqref{eq:projcuts01} can be decomposed in a family of projection cuts for each $i \in V$, since Farkas multipliers are independent for each $i$. In fact, if there is an inequality \eqref{eq:projcuts01} violated by a point $\overline{\mathbf{x}} \in [0,1]^m$, there must be a vertex $i \in V$ such that the inequality
\begin{equation}
\label{eq:projcuts03}
\sum_{e  = \{i,j\} \in \delta(i)}\beta_{ij}^ix_e    \leq \tau_i
\end{equation}
\noindent is also violated by $\overline{\mathbf{x}}$. Thus, projection cuts can be written for each vertex $i \in V$, independently, as \eqref{eq:projcuts03}. Denote by $\overline{\mathbf{x}}_i = \{\overline{x}_{ij}: \{i,j\} \in \delta(i)\}$ the piece of the vector  $\overline{\mathbf{x}}$, associated to the edges in $\delta(i)$. Likewise, define $\beta^i = \{\beta_{ij}^i: \{i,j\} \in \delta(i)\}$. For a given $i \in V$ and $\overline{\mathbf{x}}_i \in [0,1]^{|\delta(i)|}$, deciding whether or not $\overline{\mathbf{x}} \in \hbox{Proj}_x(\mathcal{F}_{xy}^*)$ amounts to solving one 
separation problem ($\hbox{{\bf SEP}}_i$), defined below, for each $i \in V$.

\begin{align}
(\hbox{{\bf SEP}}_i) \hspace{3mm} \min \sum_{e  = \{i,j\} \in \delta(i)}( - \beta_{ij}^i)\overline{x}_e + \tau_i \nonumber \\ 
(\tau_i, \beta^i) & \hbox{ satisfies } \eqref{proj1}-\eqref{proj2} \nonumber \\
%\tau_i   
% & \geq \sum_{k| \measuredangle_{jik} \leq \alpha}\beta_{ik}^i & & \{i,j\} \in \delta(i)\label{eq:sep302} \\
  \tau_i & \leq 1 & & \label{eq:norm2} 
% \beta_{ij}^i & \geq 0 & & \{i,j\} \in \delta(i) \nonumber  
 \end{align}

The separation problem ($\hbox{{\bf SEP}}_i$) includes a normalization constraint \eqref{eq:norm2}, since otherwise the problem could be unlimited. For instance, one could take a solution $\overline{\beta}^i, \overline{\tau}_i$ for which the objective function above is less than zero, and scale it by any constant larger than $1$,  resulting in feasible Farkas multipliers associated with a smaller objective function.

We now show that, given $s(i) \subseteq \delta(i), s(i) \not = \emptyset$, the Farkas multipliers defined by \eqref{eq:solutionsproj01}-\eqref{eq:solutionsproj03} are feasible to ($\hbox{{\bf SEP}}_i$), and thus generate valid projection cuts.
\begin{align}
\hat{\beta}_{ij}^i & = \frac{1}{v^{(s(i),\alpha)}} & & \{i,j\} \in s(i) \label{eq:solutionsproj01} \\
\hat{\beta}_{ij}^i & = 0                    & & \{i,j\} \in \delta(i) \setminus s(i) \label{eq:solutionsproj02} \\
\hat{\tau}_i & = 1 \label{eq:solutionsproj03}
\end{align}

Since $\hat{\beta}_{ij}^i \geq 0, \hat{\tau}_i = 1$, all we need to show is that constraints \eqref{proj1} are satisfied. For any $\{i,j\} \in \delta(i)$ we then have:
\begin{align*}
    \sum_{k| \measuredangle_{jik} \leq \alpha}\hat{\beta}_{ik}^i & = & \sum_{\{i,k\} \in s(i) | \measuredangle_{jik} \leq \alpha}\hat{\beta}_{ik}^i        + \sum_{\{i,k\} \in \delta(i)\setminus s(i)| \measuredangle_{jik} \leq \alpha}\hat{\beta}_{ik}^i                &  \\
 & = & \sum_{\{i,k\} \in s(i)| \measuredangle_{jik} \leq \alpha}\hat{\beta}_{ik}^i    &  \\
 & = & \frac{1}{v^{(s(i),\alpha)}} | \{ \{i,k\} \in s(i): \measuredangle_{jik} \leq \alpha    \}| \\
 & = & \frac{1}{v^{(s(i),\alpha)}} v_{ij}^{(s(i),\alpha)} \\
 & \leq & 1 \\
 & =     & \hat{\tau}_i
\end{align*}
Note that the inequality above applies because $v_{ij}^{(s(i),\alpha)} \leq v^{(s(i),\alpha)}$ always holds, either by the definition of $v^{(s(i),\alpha)}$, if $\{i,j\} \in s(i)$, or as a consequence of Lemma \ref{lemnovo}, if $\{i,j\} \not \in s(i)$. Multiplying the Farkas multipliers \eqref{eq:solutionsproj01}-\eqref{eq:solutionsproj03} by $v^{(s(i),\alpha)}$, the projection cut \eqref{eq:projcuts03} reads precisely as \eqref{eq:strenthen3} and the proof of Theorem \ref{theoprojection} is complete. 
\end{proof}

To conclude this appendix, note that, since ${\cal F}_{xy} \subseteq {\cal F}_{xy}^*$, $\hbox{Proj}_x(\mathcal{F}_{xy}) \subseteq \hbox{Proj}_x(\mathcal{F}_{xy}^*) \subseteq \mathcal{F}_x^+$ follows.

%===============================================
% ESTA PARTE SERA O MATERIAL ON-LINE COMPLEMENTAR
%=================================================
%\section{Detailed Computational Results}
%\label{sec:appr}
%\input{./Tabelas/lpr1.tex}
%\input{./Tabelas/lpr2.tex}
%\input{./Tabelas/lpr3.tex}
%\input{./Tabelas/BC60.tex}
%\input{./Tabelas/BC72.tex}
%\input{./Tabelas/BC90.tex}
%\input{./Tabelas/BC108.tex}
%\input{./Tabelas/BC120.tex}
%\input{./Tabelas/BC144.tex}

%==============================================
%\bibliographystyle{elsarticle-num-names}
%\bibliography{angular}

\end{document}